\documentclass[12pt,a4paper]{article}
\usepackage{amssymb}
\usepackage{amsmath,amsfonts,amsthm}

\usepackage{graphicx}
\usepackage{amsmath}

\usepackage{amsfonts}
\usepackage{amsthm,amscd}
\usepackage{graphicx}
\usepackage{amsmath}
\usepackage{amssymb}
\usepackage{amstext}

\newtheorem{theorem}{Theorem}
\newtheorem{corollary}[theorem]{Corollary}
\newtheorem{definition}[theorem]{Definition}
\newtheorem{example}[theorem]{Example}
\newtheorem{lemma}[theorem]{Lemma}
\newtheorem{notation}[theorem]{Notation}
\newtheorem{proposition}[theorem]{Proposition}

\title{ Large Deviations for Quantum Spin probabilities at temperature zero}

\author{Artur O. Lopes, Jairo K. Mengue, \\
Joana Mohr and Carlos G. Moreira (*)  \\
\\ Inst. Mat - UFRGS, Brazil and \\
(*) IMPA - Brazil}

\begin{document}

\maketitle

\begin{abstract}

We consider certain self-adjoint observable for the KMS state associated to  the Hamiltonian $H= \sigma^x \otimes \sigma^x$ over the quantum spin lattice
$\mathbb{C}^2 \otimes \mathbb{C}^2 \otimes \mathbb{C}^2  \otimes  ...$. For a fixed observable of the form $L \otimes L \otimes L  \otimes  ...$, where  $L:\mathbb{C}^2 \to \mathbb{C}^2 $, and for the zero temperature limit one can get a naturally defined stationary probability $\mu$ on
the Bernoulli space $\{1,2\}^\mathbb{N}$. This probability is ergodic but it is not mixing for the shift map. It is not a Gibbs state for a continuous normalized potential but its Jacobian assume only two values almost everywhere.  Anyway, for such probability $\mu$ we can show that a Large Deviation Principle is true for a certain class of  functions. The result is derived by showing the explicit form of  the free energy which is differentiable.

\end{abstract}

\section{Introduction}

We analyze from the point of view of Ergodic Theory a probability $\mu$ which appears in a natural way on a physical problem of quantum nature (see \cite{Pil, RMNS}). We point out that the system we considered in this work - for the construction of the associated $\mu$ - is of quantum nature, but the analysis of the associated measure is described in terms of classical Ergodic Theory and Symbolic Dynamics.

Given a selfadjoint operator $H$ acting on a finite dimensional Hilbert space $\mathcal{H}$ and a temperature $T>0$,  the density operator
$$\rho_{H,T}= \frac{e^{-\frac{1}{T}\, H }}{Z(T)}, $$
where\footnote{The notation $Tr$ means the trace of the operator} $Z(T)=$ Tr $e^{-\frac{1}{T}\, H }$, is called the KMS state associated to the Hamiltonian $H$.
Among other properties $\rho_{H,T}$ maximizes
$$ - \,\text{Tr}\,\,\, ( H\, \rho ) - \,\,\,\text{Tr}\,\,(\rho\, \log \rho) $$
among density operators $\rho$ acting on $\mathcal{H}$.
In this way the KMS state plays in Quantum Statistical Physics (see \cite{Bra}, \cite{PF} or \cite{Evans} for general results) the role of the Gibbs state in Thermodynamic Formalism. A general reference on Quantum Mechanics is \cite{GS} (see also \cite{Lop}).

 The paper \cite{RMNS} is the main motivation for the present work. We initially recall the general setting of that paper.

The set of linear operators acting on $\mathbb{C}^d$ will be denoted by $\mathcal{M}_d.$
We will call
$\omega=\omega_n:  \underbrace{\mathcal{M}_d \otimes \mathcal{M}_d \otimes...
\otimes \mathcal{M}_d}_n\to \mathbb{C}$ a $C^*$-dynamical state if
$\omega_n ( I^{ \otimes \, n})=1$ and  $\omega_n(a)\geq 0$, if $a$ is a non-negative element in the tensor product.
It follows that $\omega_n(L_1 \otimes L_2 \otimes \,...\otimes L_n)$ is a positive number if all
$L_j \in \mathcal{M}_d$ are positive, $j=1,2,...,n$.

A state  $\omega=\omega_n, \, n\geq 2$ will be defined in the following way: consider a fixed value $\beta>0$  and a complex self-adjoint operator $H$, depending on two variables,
$\,H:  (\mathbb{C}^d \otimes  \mathbb{C}^d)\,\to (\mathbb{C}^d \otimes  \mathbb{C}^d)$. Now for a given natural number  $n\geq 2$
let $ H_n= (\mathbb{C}^d)^{\otimes n} \to (\mathbb{C}^d)^{\otimes n}$ be given by
$$H_n=\sum_{j=0}^{n-2} I^{ \otimes \,j} \otimes  H \otimes  I^{ \otimes\,(n-j-2) } . $$
Set
$$ \rho_{\omega}=\rho_{ \omega_{H,\beta,n}} = \frac{1}{Tr\, ( e^{-\, \beta H_n}  ) }  e^{-\, \beta H_n},$$
and, finally, we define the state $\omega_n=\omega_{H,\beta,n}$ by
\begin{align} \label{def_omega}
\nonumber \omega_{H,\beta,n}(L_1 \otimes L_2 \otimes \,....\otimes L_n)&= \frac{1}{Tr\, ( e^{-\, \beta H_n}  ) }
\text{Tr} [\,e^{-\, \beta H_n} (L_1 \otimes L_2 \otimes \,....\otimes L_n) \,]\\
&=\text{Tr}\,\,( \rho_{\omega}\,L_1 \otimes L_2 \otimes \,....\otimes L_n).
\end{align}
\medskip

The parameter $\beta$ corresponds to $1/T$.

From the above construction, any Hamiltonian $H$ acting on $\mathbb{C}^d \otimes  \mathbb{C}^d$ defines a state $\omega=\omega_{H,\beta,n}$ acting on $\underbrace{\mathcal{M}_d \otimes \mathcal{M}_d \otimes...
	\otimes \mathcal{M}_d}_n$. A Quantum Ising Chain is defined by a Hamiltonian of the form
$$ H=- J\,(\, \sigma_1^x \,\otimes \,\sigma^x_2 \,)- h\, (\,\sigma_1^z\, \otimes\, I\,), $$
where the Pauli matrices are
$$
\sigma^x=\left(
\begin{array}{cc}
0 & 1\\
1 & 0
\end{array}
\right), \hspace{0.5cm}
\sigma^y=\left(
\begin{array}{cc}
0& -i\\
i & 0
\end{array}
\right) \hspace{0.5cm}
\text{and}
\hspace{0.5cm}
\sigma^z=\left(
\begin{array}{cc}
1 & 0\\
0 & -1
\end{array}
\right),$$
while
 $\sigma^x_i$, is the Pauli matrix $\sigma^x$ acting on the position $i$ of the tensor product.
The associated $n$-Hamiltonian is

$$H_n =  \,\sum_{i=1}^n\,[\,\,-J\,(\, \sigma_i^x \,\otimes \,\sigma^x_{i+1} \,)- h\, (\,\sigma_i^z\, \otimes\, I\,)\,\,]. $$
In the case $h=0$ we will say that the Quantum Ising Chain has no magnetic term. This will be the case here as we will see soon.

\medskip

In order to describe the construction of a naturally defined probability on the symbolic space we need some more definitions.

\medskip

 Let  $L:  \mathbb{C}^d \to \mathbb{C}^d $ be a self-adjoint operator and denote by $\lambda_1, \lambda_2,...,\lambda_d$
 its real eigenvalues.  We suppose that $\{\psi_1,...,\psi_d\}$  denotes an orthonormal
set of eigenvectors of $L$ and $P_{\psi_j}: \mathbb{C}^d \to \mathbb{C}^d$ denotes the orthogonal projection on the subspace
generated by $\psi_j$, \,$j\in\{1,...,d\}$.
For any $n\in \mathbb{N}$, the observable
$$ L^{ \otimes \, n}:=(L \otimes L \otimes \,....\otimes L): (\mathbb{C}^d \otimes \mathbb{C}^d \otimes \,....\otimes \mathbb{C}^d)
\to (\mathbb{C}^d \otimes \mathbb{C}^d \otimes \,....\otimes \mathbb{C}^d)$$
has the eigenvector $ (\, \psi_{j_1} \otimes \psi_{j_2} \otimes .... \otimes  \psi_{j_n}\,)$
associated to the eigenvalue  $\lambda_{j_1} \cdot \lambda_{j_2} \cdots  \lambda_{j_n}.$ Any eigenvalue of $ L^{ \otimes \, n}$ is of this form.

%In  the generic case all the possible eigenvalues $\lambda_{j_1}, \, \lambda_{j_2},...\,\, ,\lambda_{j_k},...\,,\lambda_{j_n}$ are different and all possible products are also different. In this case there is a bijection of the strings $\lambda_{j_1}, \, \lambda_{j_2},...\,\, ,\lambda_{j_k},...\,,\lambda_{j_n}$ an eigenvalues of the operator acting on the tensor product. But this is not an essential issue.

The values obtained by physical measuring (associated to the observable $ L^{ \otimes \, n}$) in the finite dimensional Quantum Mechanics setting
are (of course) the eigenvalues of $ L^{ \otimes \, n}$. The relevant information is the probability of each possible outcome event of measuring $ L^{ \otimes \, n}$.

For fixed $H: \mathbb{C}^d \otimes \mathbb{C}^d\to  \mathbb{C}^d \otimes \mathbb{C}^d$, $L: \mathbb{C}^d \to \mathbb{C}^d$,  %we can define a probability $\mu_\beta$ on
%$\Omega =\,\{\lambda_1, \lambda_2,...,\lambda_d\}^\mathbb{N} $. This will be done by defining
%$\mu$ on cylinders of size $n$. That is, for each $n$, we will define
%$\mu_{\beta,n}$ over $\Omega_n =\,\{\lambda_1, \lambda_2,...,\lambda_d\}^n $
%and then we use Kolmogorov extension theorem to get $\mu_\beta$ on $\Omega$.
%
%The zero temperature case is concerned with the limit of $\mu_\beta$, when $\beta \to \infty$.
%
 $n\in\{2,3,...\}$ and $\beta>0$,  we consider the $C^*$-state $\omega_{H,\beta,n}$
given in \eqref{def_omega},
and define the probability $\mu_{\beta,n}$ on $\Omega_{n} =\,\{1, 2,...,d\}^n $
by
\begin{align*}
\mu_{\beta,n} ( \,{j_1} ,  ...,
 {j_n}\,  ) &=\omega_{H,\beta,n}(  \, P_{\psi_{j_1}} \otimes  P_{\psi_{j_2}}\otimes...\otimes  P_{\psi_{j_n}}\,)\\
&=\frac{1}{Tr\, ( e^{-\, \beta H_n}  ) }\text{Tr} [\,e^{-\, \beta H_n} ( P_{\psi_{j_1}} \otimes  P_{\psi_{j_2}}\otimes...\otimes  P_{\psi_{j_n}}).
\end{align*}

We will consider here the case where $d=2$. We denote $\Omega = \{1,2\}^{\mathbb{N}}$ and for fixed $\beta$ with the information of the probabilities
$\mu_{\beta,n}$, $n \in \mathbb{N}$ one can produce a probability $\mu_{\beta}$ on $\Omega$.

 On the zero temperature limit case we consider the probabilities (if the limit exists) $$\mu_{\infty,n}=\lim_{\beta \to +\infty}\mu_{\beta,n}.$$

With this information one can get the probability of each possible outcome event of measuring via $ L^{ \otimes \, n}$. Indeed, for fixed $n$ and a specific eigenvalue $\lambda$ of $ L^{ \otimes \, n}$ we just collect all the $n$-strings $({j_1} ,  {j_2},...,{j_{n-1}}, {j_n})\, $ such that $\lambda =\lambda_{j_1} \cdot \lambda_{j_2} \cdots  \lambda_{j_n}$ and add their individual  probabilities.

For instance, in the case $d=2$, the eigenvalues of $L$ are  the numbers $\lambda_1=1$ and $\lambda_2=-1$,  then the possible eigenvalues of $ L^{ \otimes \, n}$ are also just the values $1$ and $-1$. The $3$-strings producing the value $\lambda=-1$ are
\[ \{(j_1,j_2,j_3) \,|\, \lambda_{j_1}\cdot\lambda_{j_2}\cdot\lambda_{j_3} =-1 \} =
\{(1,1,2),(1,2,1),(2,1,1), (2,2,2)  \}.
 \]

{\bf Assumption A:} In this work we consider the Quantum spin Chain with no magnetic term defined by the Hamiltonian
$H=\sigma_1^x \,\otimes \,\sigma^x_2$.  We consider self-adjoint operators (observable) of the form

$$L=\, \left(
\begin{array}{cc}
\cos^2(\theta)-\sin^2(\theta) & 2\cos(\theta)\sin(\theta)\\
2\cos(\theta)\sin(\theta) & \sin^2(\theta)-\cos^2(\theta)
\end{array}
\right),$$
where\footnote{The case $\theta =0$ corresponds to $L=\sigma^z$ and the case $\theta = \pi/2$ corresponds to $L=-\sigma^z$. We will exclude these cases.}
	  $\theta\in(0,\pi/2)$.

\bigskip

%The  corresponding associated observable
%$$ C_n=C= \underbrace{L \otimes  L \otimes  L \otimes... \otimes   L}_n.$$

The eigenvalues of $L$ are $\lambda_1=1$ and $\lambda_2=-1$ which are associated, respectively, to the unitary eigenvectors $\psi_1= (\cos(\theta), \sin(\theta))\in \mathbb{C}^2 $ and $\psi_{2} =(-\sin(\theta),\cos(\theta))\in\mathbb{C}^2$, which are orthogonal.

We are mainly interested in the zero temperature limit case. Our main result is the following:

\begin{theorem}\label{teo1}
Suppose that $H$ and $L$ satisfy  Assumption A. Consider the  probability measure $\mu$ on $\Omega = \{1,2\}^{\mathbb{N}}$, such that, for any $n\geq 2$ and $\{j_1,...,j_n\} \in \{1,2\}^n$, $\mu([j_1,...,j_n])=\mu_{\infty,n}(j_1,...,j_n)$, where $[j_1,...,j_n]=\{(x_1,x_2,...)\in \Omega\,|\,x_1=j_1,...,x_n=j_n\}$. Then, the probability $\mu$ is invariant and ergodic for the shift map $\sigma$ acting in $\Omega$. It is not mixing and does not have a continuous Jacobian.
\end{theorem} 	
\medskip

The proof of the above result is the purpose of sections 2, 3 and 4.

One can also show that this Jacobian  assume $\mu$-almost everywhere only the values $p=\frac{1-\sin(2\theta)}{2}$ and $1-p=\frac{1+\sin(2\theta)}{2}$. This will establish an upper and lower bounds for its entropy.
\medskip

\medskip

Some technical results that we need on sections 2, 3 and 4  are proved on the Appendix.

The proofs of our  main results are rigorous and they will be obtained by lengthy estimations and long induction procedures from which we will finally  derive explicit expressions. We believe it is worthwhile for the future study of measures obtained from more general Hamiltonians  (with magnetic term for instance) to  be able to establish what good dynamical properties one can expect (and which ones one can not  expect).

\medskip

 We say that there exists a {\bf Large Deviation Principle} (LDP for short) for the probability $\mu$ on $\Omega$ and the function $A: \Omega \to \mathbb{R}$, if there exist a lower-semicontinuous function $I:\mathbb{R} \to \mathbb{R}$, such that,
 \medskip

 a) for all closed sets $K \subset \mathbb{R}$ we get
$$ \lim_{n \to \infty} \frac{1}{n} \log  \bigg(\mu\,\,\bigg\{ z  \,\text{such that}\,\,\,\, {1 \over n} \sum_{j=0}^{n-1} A(\sigma ^{j} (z)) \in K\, \bigg\}\,\, \bigg) \leq \, - \inf_{s \in K} \, I(s),$$

b) for all open  sets $B \subset \mathbb{R}$ we get
$$ \lim_{n \to \infty} \frac{1}{n} \log  \bigg(\mu\,\,\bigg\{ z  \,\text{such that}\,\,\,\, {1 \over n} \sum_{j=0}^{n-1} A(\sigma ^{j} (z)) \in B\, \bigg\}\,\, \bigg)\geq \,- \inf_{s \in B} \, I(s).$$

\medskip

$I$ is called the deviation function.
\medskip

Item a) can be proved in great generality due to  Chebyshev's inequality (see \cite{El} or section 5 in \cite{Lo3}).

Item b) is more difficult to get. It can be shown to be true in the case the free energy (to be defined below) is differentiable (see \cite{El} or section 5 in \cite{Lo3}).

\medskip

We refer the reader to \cite{Leb}, \cite{Len}, \cite{Hi} and \cite{Oga} for several results on the topic of Large Deviations for Quantum Spin Systems.

\medskip

 Here we will be prove the following result:

\begin{theorem}\label{teo2}
Suppose that $H$ and $L$ satisfy the Assumption A and let $\mu$ be the probability defined on Theorem \ref{teo1}. In the case $A:\Omega \to \mathbb{R}$ satisfies $A(x_1,x_2,...)=A(x_1)$ (it depends just on the first coordinate on $\Omega$), there exists a large deviation principle for $\mu$ and $A$.  This follows from the fact that the free energy
$$c(t)=  \lim_{n \to \infty}  \frac{1}{n}\,\, \log \int e^{t\,(\,A(x) + A(\sigma(x)) + A(\sigma^{2} (x)) + ...
	+ A(\sigma^{n-1} (x))\, \,)} d \mu (x)$$
exists, and is given  by the expression
\[c(t)= \frac{1}{2}\log \bigg(\Big[ e^{t\, A(1)}+e^{t\,A(2)}\Big]^{2}-\sin^2(2\theta)\Big[ e^{t\, A(1)}- e^{t\, A(2)}\Big]^{2}\bigg)-\log(2).\]
\end{theorem}

\bigskip
The proof of this Theorem is on section \ref{LDPs}.
We point out that the above function $c(t)$ (the free energy) is differentiable on $t$. From this will follow item b) above. The deviation function $I$ is the Legendre Transform of $c(t)$ (see \cite{El})

\medskip

We would like to thanks A. Quas for some useful comments which help us  in the proof of Theorem \ref{qq}.

\section{Properties of the associated probability $\mu$}

\medskip

Initially, let us to compute $U=e^{i z\, H}$ for $H= \sigma^x \otimes \sigma^x$ and for any complex number $z$. The relation $\sigma^x \circ \sigma^x= I$ is very helpful.
Note that
\[U=e^{i\,z\, (\sigma^x \otimes \sigma^x)}= \sum_{j=0}^\infty \frac{(iz)^j}{j\, !}\, ( \sigma^x \otimes \sigma^x )^j=
\sum_{j=0}^\infty \frac{(iz)^j}{j\, !}\, ( (\sigma^x)^j \otimes (\sigma^x)^j) \]
\[=\cos(z) \, (I \otimes I)\,+\,i\, \sin(z)\, (\sigma^x \otimes \sigma^x).\]

In Quantum Statistical Mechanics (see \cite{Bra}) it is natural to take $z = i \beta,$ where $\beta$ is real.
%We now want to compute  $B:=e^{- \,\beta\, H}$, $\beta $ real.
In this case
$$B=e^{-\,\beta\, \sigma^x \otimes \sigma^x}= \cos(i\, \beta) \, (I \otimes I)\,+\,i\, \sin(i\,\beta)\, (\sigma^x \otimes \sigma^x).$$

We will use the following notation:
\begin{notation} For fixed $n$,
\[(\sigma^x_{i}\otimes \sigma^x_{i+1})_n =\left\{	
\begin{array}{ll}
\underbrace{I  \otimes... \otimes I}_{i-1}\otimes \sigma^x\otimes \sigma^x \otimes \underbrace{I  \otimes... \otimes I}_{n-i-1}, & i\in\{2,...,n-2\}\\
\sigma^x\otimes \sigma^x \otimes \underbrace{I  \otimes... \otimes I}_{n-2},& i=1\\
\underbrace{I  \otimes... \otimes I}_{n-2}\otimes \sigma^x\otimes \sigma^x,& i=n-1
\end{array}\right.
\]
\end{notation}

Now, we compute $e^{-\beta H_n}:(\mathbb{C}^2)^{ \otimes n} \to (\mathbb{C}^2)^{\otimes n}$. As $(\sigma^x_{i}\otimes \sigma^x_{i+1})_n$ commutes with $(\sigma^x_{j}\otimes \sigma^x_{j+1})_n$, $i,j \in \{1,...,n-1\}$,
%
%In order to help the reader we will present some examples of the general calculation.
%As a particular case, note that $\sigma^x \otimes \sigma^x \otimes I \otimes I$ commutes with $I\otimes \sigma^x \otimes \sigma^x\otimes I,$ etc.
%In this case
we get
\begin{align*}
 e^{-\beta H_n} &= e^{-\, \beta\,[\sum_{i=1}^{n-1} (\sigma^x_{i}\otimes \sigma^x_{i+1})_n]}=\prod_{i=1}^{n-1}e^{ -\, \beta\,(\sigma^x_{i}\otimes \sigma^x_{i+1})_n}\\
 &= \prod_{i=1}^{n-1} [\cos(i \beta) \, I^{\otimes n} \,+\,i\, \sin(i \beta)\, (\sigma^x_{i}\otimes \sigma^x_{i+1})_n],
\end{align*}
where the above product represents the composition of the operators.

If $n=4$, for instance, we get

$$ e^{-\beta H_4}= e^{-\, \beta\,[\, (\sigma^x \otimes \sigma^x \otimes I \otimes I) + (I \otimes \sigma^x \otimes \sigma^x \otimes I)+
(I \otimes I \otimes \sigma^x \otimes \sigma^x)\,]}=$$
$$e^{ -\, \beta\,(\sigma^x \otimes \sigma^x \otimes I \otimes I)} \,\circ e^{-\, \beta\, (I \otimes \sigma^x \otimes \sigma^x \otimes I)}
\, \circ \,e^{-\, \beta\, (I \otimes I \otimes \sigma^x \otimes \sigma^x) }=$$
$$ [\, \cos(i \beta) \, (I \otimes I \otimes I \otimes  I)\,+\,i\, \sin(i \beta)\, (\sigma^x \otimes \sigma^x  \otimes I \otimes  I)\,]\, \circ $$
$$ [\, \cos(i \beta) \, (I \otimes I \otimes I \otimes  I)\,+\,i\, \sin(i \beta)\, (I \otimes \sigma^x \otimes \sigma^x  \otimes I)\,]\, \circ $$
 \begin{equation} \label{ooe}[\, \cos(i \beta) \, (I \otimes I \otimes I \otimes  I)\,+\,i\, \sin(i \beta)\, (I \otimes I \otimes \sigma^x \otimes \sigma^x)\,].\,
 \end{equation}

\begin{lemma}\label{tr} If $H=\sigma^x\otimes\sigma^x$, then ${\text{Tr}(e^{-\beta\,H_n})}={\cos^{n-1}(i \beta) 2^n}$.
\end{lemma}
\begin{proof} Note that
		
$$
e^{-\beta H_n}=  \prod_{i=1}^{n-1} [\cos(i \beta) \, I^{\otimes n} \,+\,i\, \sin(i \beta)\, (\sigma^x_{i}\otimes \sigma^x_{i+1})_n]$$
which results in a sum with $2^{n-1}$ terms.
Only the term $ \cos^{n-1}(i \beta)\, ( I^{\otimes n})$ do not contain a product of $\sigma^x$. As Tr $(\sigma^x)=0$ and
$$Tr (L_1\otimes...\otimes L_n)=Tr(L_1)\cdots Tr(L_n),$$
we will get that any term which does  not contain the expression $\cos^{n-1}(i \beta)\, ( I^{\otimes n})$ will produce a null trace. Moreover, as Tr $(I)=2$, we finally get
$$\text{Tr}(e^{-\beta H_n})= \cos^{n-1}(i \beta) Tr ( I^{\otimes n})=  \cos^{n-1}(i \beta) 2^n.$$

\end{proof}

\bigskip

From now on we will consider that $L$ satisfies the Assumption A, that is,
$$L=\, \left(
\begin{array}{cc}
\cos^2(\theta)-\sin^2(\theta) & 2\cos(\theta)\sin(\theta)\\
2\cos(\theta)\sin(\theta) & \sin^2(\theta)-\cos^2(\theta)
\end{array}
\right),$$
where $\theta\in(0,\pi/2)$,
and we will analyze the  corresponding associated observable
$$ L^{\otimes n}= \underbrace{L \otimes  L \otimes  L \otimes... \otimes   L}_n.$$

The eigenvalues of $L$ are $\lambda_1=1$ and $\lambda_2=-1$,  associated, respectively, to the unitary eigenvectors $v_1= (\cos(\theta), \sin(\theta))\in \mathbb{C}^2 $ and $v_{2} =(-\sin(\theta),\cos(\theta))\in\mathbb{C}^2$ - which are orthogonal.
The eigenvectors of $L^{\otimes n}$ are of the form
$$ v_{j_1} \otimes v_{ j_2} \otimes  v_{ j_3} \otimes  v_{ j_4}\otimes ...\otimes v_{j_n},$$
where, $ (j_1,...,j_n) \in \{1,2\}^n.$

We denote by $P_1: \mathbb{C}^2 \to \mathbb{C}^2 $ the projection on $v_1$ and  $P_2: \mathbb{C}^2 \to \mathbb{C}^2 $
the projection on $v_2$.
In this way
$$
P_1= \, \left(
\begin{array}{cc}
\cos(\theta)^2 & \cos(\theta)\, \sin(\theta)\\
\cos(\theta)\, \sin(\theta) & \sin(\theta)^2
\end{array}
\right)$$
and
$$
P_2=  \left(
\begin{array}{cc}
\sin(\theta)^2  & -\cos(\theta)\, \sin(\theta)\\
- \cos(\theta)\, \sin(\theta) & \cos(\theta)^2
\end{array}
\right) .$$

Note that Tr $P_1=$  Tr $P_2= 1.$
Moreover,
$$\sigma^x (P_1) =\left(
\begin{array}{cc}
\cos(\theta)\sin(\theta) & \sin^2(\theta)\\
\cos^2(\theta) & \cos(\theta)\sin(\theta)
\end{array}
\right)
$$
has trace $\beta_1:= \sin(2\, \theta)\in \mathbb{R} $ and
$$\sigma^x (P_2) =
\, \left(
\begin{array}{cc}
-\cos(\theta)\sin(\theta) & \cos^2(\theta) \\
\sin^2(\theta)& -\cos(\theta)\sin(\theta)
\end{array}
\right)$$
has trace $  \beta_2:=- \sin(2\, \theta)\in \mathbb{R}$.
Therefore, Tr
$(\sigma^x (P_2))=\beta_2= - \beta_1$.
Note that, if $\theta\neq \frac{\pi}{4}$, then $\beta_1,\beta_2$ both have modulus smaller  than $1$.

\begin{notation}
	\[\beta_1:= \sin(2\theta) \hspace{1cm} \text{and} \hspace{1cm} \beta_2:= -\sin(2\theta) \]
\end{notation}

\medskip
%We will consider now $\mu_{{\beta},n}$, in each cylinder of size $n$.
%We will see below that
% $\mu_{\beta,n}$ and $\mu_{\beta,n+1}$ are compatible in the sense that \footnote{agora a prova esta so para o caso de temperatura zero, talvez mudar esta frase.}
%$$ \mu_{\beta,n}  (j_1,j_2,j_3,..,j_n) = \mu_{\beta,n+1}  (j_1,j_2,j_3,..,j_n,1)+  \mu_{\beta,n+1}  (j_1,j_2,j_3,..,j_n,2).$$
%Then, by Kolmogorov extension theorem we get a probability $\mu_{\beta}$ on the
%Bernoulli space
%$\{1,2\}^\mathbb{N}$. In the following we write $\mu_\beta$ instead of $\mu_{\beta,n}$ at some moments.
The probability $\mu_{\beta,n}$ of the element $(j_1,...,j_n) \in \{1,2\}^n$   is given by
$$\mu_{\beta,n}( j_1,,...,j_n )=\frac{1}{Tr\, ( e^{-\, \beta H_n}  ) }\text{Tr} [\,e^{-\, \beta H_n} ( P_{{j_1}} \otimes  P_{{j_2}}\otimes...\otimes  P_{{j_n}})$$
$$=\frac{1}{\cos^{n-1}(i\beta)2^n}\text{Tr}\prod_{i=1}^{n-1} [\cos(i \beta) \, I^{\otimes n} \,+\,i\, \sin(i \beta)\, (\sigma^x_{i}\otimes \sigma^x_{i+1})_n](   P_{j_1} \otimes  P_{j_2}\,\otimes...\otimes P_{j_n}).$$

Using the relation $i\,\sin(\beta i)= - \cos(\beta i)+e^{-\beta} $, when $\beta\to\infty$ we get the following result

\begin{lemma} \label{mu} If $H$ and $L$ satisfy the Assumption A, then for $n\in\{2,3,...\}$,
\begin{equation*}
\mu_{\infty,n}( {j_1} ,  {j_2},..., {j_n}  )
=\frac{1}{2^n}\text{Tr}\left[\prod_{i=1}^{n-1} [\, I^{\otimes n} \,-\, (\sigma^x_{i}\otimes \sigma^x_{i+1})_n](   P_{j_1} \otimes  P_{j_2}\,\otimes...\otimes P_{j_n})\right].
\end{equation*}
\end{lemma}

Taking the limit $\beta \to 0$ will simplify a lot the future computations.

For the benefit of the reader we consider the cases n=2 and n=3. These estimations are important for the induction procedure we will consider later.

\begin{example}\label{example}
As Tr$(\sigma^xP_{j_i})=\beta_{j_i}$, we get
 $$ \mu_{\infty,2} (j_1,j_2)=\frac{1}{2^2}\text{Tr}\bigg[(I\otimes I-\sigma^x\otimes\sigma^x)\circ(P_{j_1}\otimes P_{j_2})\bigg]= $$
 $$=\frac{1}{2^2}\text{Tr}\bigg[P_{j_1}\otimes P_{j_2}-\sigma^x P_{j_1}\otimes\sigma^xP_{j_2}\bigg] = \frac{1}{2^2}\bigg[1-\beta_{j_1}\beta_{j_2}\bigg].$$

 \bigskip

Note that if $\theta \neq \frac{\pi}{4}$, the number
$ 1 - \,\beta_{j_1}\, \beta_{j_2} $
is  positive because   $|\beta_{1}|<1$ and    $|\beta_{2}|<1$.

%\medskip
%Moreover $\mu (1) = \mu_2(1,1) +\mu_2(1,2)=1/2=\mu(2)$ for getting $\mu$ an invariant probability for the shift (see Corollary \ref{sta}).
%\medskip

We also get

 $$ \mu_{\infty,3} (j_1,j_2,j_3)=\frac{1}{2^3}\text{Tr}\,\left\{ \begin{array}{l}( \, ({I  \otimes I \otimes I})\,-
\, (\sigma^x \otimes \sigma^x  \otimes I ))\,\circ \\
( \, ({I \otimes I \otimes  I})-\, (I \otimes \sigma^x \otimes \sigma^x )\, \circ \\
 (   P_{j_1} \otimes  P_{j_2}\,\otimes P_{j_3})\,\,\end{array}\right\}= $$
 $$=\frac{1}{2^3}\bigg[1- \beta_{j_1}\beta_{j_2}-\beta_{j_2}\beta_{j_3}+\beta_{j_1}\beta_{j_3}\bigg]. $$

\end{example}

\medskip

\begin{proposition} \label{stat} For $n\in \{2,3,...\}$,
\[\mu_{\infty,n+1} (1 , j_1 , j_3,...,j_n)+ \mu_{\infty,n+1} (2, j_1 , j_3,..,j_n)= \mu_{\infty,n} ( j_1 , j_3,..,j_n).\]

\end{proposition}

\bigskip

The proof of this Proposition is on the Appendix (see Proposition \ref{stat1}).

\medskip

%The induced probability $\mu$ at temperature zero on the Bernoulli space  $\{1,2\}^\mathbb{N}$ is stationary.
As \[\mu_{\infty,n}( {j_1} ,  {j_2},..., {j_n}  )
=\frac{1}{2^n}\text{Tr}\left[\prod_{i=1}^{n-1} [\, I^{\otimes n} \,-\, (\sigma^x_{i}\otimes \sigma^x_{i+1})_n](   P_{j_1} \otimes  P_{j_2}\,\otimes...\otimes P_{j_n})\right],\]
it follows from the commutativity of the composition of the several terms above that
\begin{equation} \label{uu} \mu_{\infty,n} ( j_1,j_2,j_3,j_4,..,j_n )=  \mu_{\infty,n} ( j_n,j_{n-1},j_{n-2},...j_2,j_1 ).
\end{equation}

\begin{corollary} \label{sta}
	Suppose that $H$ and $L$ satisfy the Assumption A. There exists a probability $\mu$ on $\{1,2\}^{\mathbb{N}}$ invariant for the shift map $\sigma$, such that, for any $n \in \{2,3,4,...\}$ and any cylinder $[j_1,...,j_n]$, we get $\mu([j_1,...,j_n])=\mu_{\infty,n}(j_1,...,j_n)$.
	\end{corollary}	

\begin{proof} We assume $\mu_{\infty,1}(1)=\mu_{\infty,1}(2)=\frac{1}{2}$ and we observe that $$\mu_{\infty,1}(j)=\sum_{i=1}^2 \mu_{\infty,2}(i,j)=\sum_{j=1}^2 \mu_{\infty,2}(i,j)$$ (see example \ref{example}).
For any $n\in\mathbb{N}$, the probabilities $\mu_{\infty,n}$ and $\mu_{\infty,n+1}$ are compatible in the sense that
$$ \mu_{\infty,n}  (j_1,j_2,j_3,..,j_n) = \mu_{\infty,n+1}  (j_1,j_2,j_3,..,j_n,1)+  \mu_{\infty,n+1}  (j_1,j_2,j_3,..,j_n,2).$$
This  follows from (\ref{uu}) and Proposition \ref{stat}.
In this way by Kolmogorov extension Theorem we get a probability $\mu$, on the
Bernoulli space $\{1,2\}^\mathbb{N},$  satisfying $\mu([j_1,...,j_n])=\mu_{\infty,n}(j_1,...,j_n)\,, \forall\,n \in \mathbb{N}$.
It is invariant for the shift map $\sigma$ from Proposition \ref{stat} - which is true also for $n=1$ (using the above definition of $\mu_{\infty,1}$).
\end{proof}

\begin{notation}
 $\mu(j_1,...,j_n)=\mu([j_1,...,j_n])$.
\end{notation}

\bigskip

We will finish this section by showing how to compute $\mu$ recursively in cylinders. The next result will be very helpful on the next sections.

\begin{theorem}\label{medidacilindro}
Suppose $H$ and $L$ satisfy the Assumption A. Let $\mu$ be the probability measure defined in {corollary} \ref{sta}. For any $n\geq 2$, we have
\small
\begin{equation}\label{eq0}
\mu(k,j_1,...,j_n)= \frac{\mu( j_1,...,j_n )}{2}  + \sum_{i=1}^{n-2}\frac{(-1)^i \beta_k \beta_{j_i} }{2^{i+1}}\mu(j_{i+1},...,j_n)  +\frac{(-1)^n\beta_k(\beta_{j_n}-\beta_{j_{n-1}})}{2^{n+1}} .
\end{equation}
\normalsize
and
\[\mu(k,j_1)= \frac{\mu(j_1)}{2} -\frac{\beta_k\beta_{j_1}}{4}.\]
\end{theorem}

\medskip

The proof of the above Theorem is on the Appendix (see Theorem \ref{medidacilindro1}).

\bigskip

One more example will be presented.

\begin{example}\label{ex}

\medskip

$$\mu(k,j_1,j_2) = \frac{\mu(j_1,j_2)}{2} - \frac{\beta_k \beta_{j_{1}}}{8} +  \frac{\beta_k \beta_{j_2}}{8},$$

$$\mu(k,j_1,j_2,j_3) = \frac{\mu(j_1,j_2,j_3)}{2} - \frac{\beta_k \beta_{j_1}}{4} \mu(j_2,j_3) +  \frac{\beta_k \beta_{j_2}}{16}-  \frac{\beta_k \beta_{j_3}}{16},$$

\begin{align*}
\mu(k,j_1,j_2,j_3,j_4) &= \frac{\mu(j_1,j_2,j_3,j_4)}{2} - \frac{\beta_k \beta_{j_1}}{4} \mu(j_2,j_3,j_4)\\
&\,\,\,\, +\frac{  \beta_k \beta_{j_2}}{8} \mu(j_3,j_4)-  \frac{\beta_k \beta_{j_3}}{32}+ \frac{\beta_k \beta_{j_4}}{32}.
\end{align*}

\end{example}

\section{$\mu$ is ergodic but it is not mixing}
In this section we suppose again that $H$ and $L$ satisfy the Assumption A.
We will show that the probability $\mu$ is {\bf ergodic but it is not mixing}.

\medskip

\begin{proposition} \label{nmix} For any $n\geq 2$ we have that
$$\sum_{j_2,..,j_n}\mu (a ,b , j_2,...,j_n,c,d)=\mu(a,b)\mu(c,d )+\frac{(-1)^{n}( \beta_{b}-\beta_a) (\beta_{c}-\beta_{d})}{2^{4}}. $$
Particularly, $\mu$ is not mixing.
\end{proposition}

\medskip

The proof of this Proposition is on the Appendix (see Proposition \ref{nmix1}).

\medskip

Now we will prove that $\mu$ is ergodic. Initially, we need a preliminary result:

 \begin{proposition}\label{recursive}
 	For $n\geq 1$,
 	$$\mu (k_0 , k_{1} , ...,k_n)=\frac{1}{2}\bigg(1-\frac{\beta_{k_0}}{\beta_{k_1}}\bigg)\mu (k_1 , ...,k_{n}) +\frac{\beta_{k_0}}{2} \bigg(\frac{1}{2\beta_{k_{1}}}-\frac{\beta_{k_{1}}}{2}\bigg)\mu (k_2,...,k_{n}).$$
  \end{proposition}

 \medskip

The proof of this claim is on the Appendix (see Proposition \ref{recursive1})

\medskip
 \begin{example}
 We remember that $\beta_1 = \sin(2\theta)$ and $\beta_2 = -\beta_1$. In the particular case that $\theta=\pi/4$ we have $\beta_1=1$, $\beta_2 = -1$ and, using the above Proposition, we get

 $$\mu (k_0 , k_{1} , ...,k_n)=\frac{1}{2}\bigg(1-\frac{\beta_{k_0}}{\beta_{k_{1}}}\bigg)\mu (k_1 , ...,k_{n}).$$
 If $k_0=k_{1}$, then $\mu (k_0 , k_{1} , ...,k_n)=0$. We conclude that $\mu$ is supported in the periodic orbit of the point $(0,1,0,1,0,1,...)$.

\end{example}

\begin{theorem}
For any cylinders $[a_1,...,a_k]$ and $[b_1,...,b_l]$ we have
\begin{equation}\label{erg}
\lim_{N}\frac{1}{N}\sum_{n=1}^{N}\sum_{j_1,...,j_n}\mu(a_1,...,a_k,j_1,...,j_n,b_1,...,b_l)=\mu(a_1,...,a_k)\mu(b_1,...,b_l).  \end{equation}
Particularly, $\mu$ is ergodic.
\end{theorem}
\begin{proof}
We remark that for $A=[a_1,...,a_k]$ and $B=[b_1,...,b_l]$, equation (\ref{erg})
means
\[\lim_{N}\frac{1}{N}\sum_{n=k+1}^{k+N}\mu(A\cap\sigma^{-n}B)=\mu(A)\mu(B).    \]
It is easy to extend the above expression for any measurable sets. This will show that $\mu$ is ergodic.

We will prove (\ref{erg}) by induction on $k$.
For $k=1$, from Theorem \ref{medidacilindro}, as $\beta_2=-\beta_1$, we get
\begin{align*}
\sum_{j_1,...,j_n}\mu(a_1,j_1,...,j_n,b_1,...,b_l) &= \frac{\mu(b_1,...,b_l)}{2} + \sum_{i=1}^{l-2}\frac{(-1)^{n+i}\beta_{a_1}\beta_{b_i} }{2^{i+1}}\mu(b_{i+1},...,b_l)\\
&\,\,\,+ \frac{(-1)^{n+l-1}\beta_{a_1}\beta_{b_{l-1}}}{2^{l+1}} + \frac{(-1)^{n+l}\beta_{a_1}\beta_{b_{l}}}{2^{l+1}}.
\end{align*}
Then, there exists a constant C, such that
\[\sum_{j_1,...,j_n}\mu(a_1,j_1,...,j_n,b_1,...,b_l) = \frac{\mu(b_1,...,b_l)}{2}+(-1)^nC.\]
Therefore,
\[\lim_{N}\frac{1}{N}\sum_{n=1}^{N}\sum_{j_1,...,j_n}\mu(a_1,j_1,...,j_n,b_1,...,b_l)=\frac{\mu(b_1,...,b_l)}{2}=\mu(a_1)\mu(b_1,...,b_l).\]
For $k=2$ we have, 	
\begin{align*}
\sum_{j_1,...,j_n}&\mu(a_1,a_2,j_1,...,j_n,b_1,...,b_l) \\
&= \sum_{j_1,...,j_n}\frac{\mu(a_2,j_1,...j_n,b_1,...,b_l)}{2}\\
&\,\,\, -\frac{\beta_{a_1}\beta_{a_2}}{4}\mu(b_1,...,b_l)
+ \sum_{i=1}^{l-2}\frac{(-1)^{1+n+i}\beta_{a_1}\beta_{b_i} }{2^{i+2}}\mu(b_{i+1},...,b_l) \\
&\,\, + \frac{(-1)^{n+l}\beta_{a_1}\beta_{b_{l-1}}}{2^{l+2}} + \frac{(-1)^{1+n+l}\beta_{a_1}\beta_{b_{l}}}{2^{l+2}}.
\end{align*}
Using similar computations as above for the case  $k=1$  we get
\begin{align*}
\sum_{j_1,...,j_n}&\mu(a_1,a_2,j_1,...,j_n,b_1,...,b_l) \\
&= \frac{\mu(b_1,...,b_l)}{4} + \sum_{i=1}^{l-2}\frac{(-1)^{n+i}\beta_{a_2}\beta_{b_i} }{2^{i+2}}\mu(b_{i+1},...,b_l)\\
&\,\,\,+ \frac{(-1)^{n+l-1}\beta_{a_2}\beta_{b_{l-1}}}{2^{l+2}} + \frac{(-1)^{n+l}\beta_{a_2}\beta_{b_{l}}}{2^{l+2}}\\
&-\frac{\beta_{a_1}\beta_{a_2}}{4}\mu(b_1,...,b_l)+ \sum_{i=1}^{l-2}\frac{(-1)^{1+n+i}\beta_{a_1}\beta_{b_i} }{2^{i+2}}\mu(b_{i+1},...,b_l) \\
&\,\, + \frac{(-1)^{n+l}\beta_{a_1}\beta_{b_{l-1}}}{2^{l+2}} + \frac{(-1)^{1+n+l}\beta_{a_1}\beta_{b_{l}}}{2^{l+2}},
\end{align*}
that is, there exists a constant C, such that
\[\sum_{j_1,...,j_n}\mu(a_1,a_2,j_1,...,j_n,b_1,...,b_l)=\left(\frac{1}{4}-\frac{\beta_{a_1}\beta_{a_2}}{4}\right)\mu(b_1,...,b_l)+(-1)^n C.  \]
Therefore, as from example \ref{example} we have that $\left(\frac{1}{4}-\frac{\beta_{a_1}\beta_{a_2}}{4}\right)=\mu(a_1,a_2)$, then,
\[\lim_{N}\frac{1}{N}\sum_{n=1}^{N}\sum_{j_1,...,j_n}\mu(a_1,a_2,j_1,...,j_n,b_1,...,b_l)=\mu(a_1,a_2)\mu(b_1,...,b_l).\]

Now, we suppose that for a fixed $k_0$ and any $a_1,...,a_k$, $k\leq k_0$, the equation (\ref{erg}) holds. We want to prove that for any $a\in\{1,2\}$, and any $a_1,...,a_k \in\{1,2\}^{k},\,\,2\leq k\leq k_0,$ we get
\[\lim_{N}\frac{1}{N}\sum_{n=1}^{N}\sum_{j_1,...,j_n}\mu(a,a_1,a_2,...,a_k,j_1,...,j_n,b_1,...,b_l)=\mu(a,a_1,...,a_k)\mu(b_1,...,b_l).\]
From Proposition \ref{recursive},
\begin{align*}
\sum_{j_1,...,j_n}&\mu(a,a_1,a_2,...,a_k,j_1,...,j_n,b_1,...,b_l)\\
&= \sum_{j_1,...,j_n}  \frac{1}{2}\bigg(1-\frac{\beta_{a}}{\beta_{a_1}}\bigg)\mu (a_1,...,a_k,j_1,...,j_l,b_1, ...,b_l)\\
&\,\,\,\,\,+\sum_{j_1,...,j_n}\frac{\beta_{a}}{2} \bigg(\frac{1}{2\beta_{a_{1}}}-\frac{\beta_{a_{1}}}{2}\bigg)\mu (a_2,...,a_k,j_1,...,j_l,b_1, ...,b_l).
 \end{align*}
Using the induction hypothesis and Proposition \ref{recursive} again we get
\begin{align*}
\sum_{j_1,...,j_n}&\mu(a,a_1,a_2,...,a_k,j_1,...,j_n,b_1,...,b_l)\\
&=   \frac{1}{2}\bigg(1-\frac{\beta_{a}}{\beta_{a_1}}\bigg)\mu (a_1,...,a_k)\mu(b_1, ...,b_l)\\
&\,\,\,\,\,+\frac{\beta_{a}}{2} \bigg(\frac{1}{2\beta_{a_{1}}}-\frac{\beta_{a_{1}}}{2}\bigg)\mu (a_2,...,a_k)\mu(b_1, ...,b_l) \\
&= \left[\frac{1}{2}\bigg(1-\frac{\beta_{a}}{\beta_{a_1}}\bigg)\mu (a_1,...,a_k)+\frac{\beta_{a}}{2} \bigg(\frac{1}{2\beta_{a_{1}}}-\frac{\beta_{a_{1}}}{2}\bigg)\mu (a_2,...,a_k)\right]\mu(b_1, ...,b_l)\\
&=[\mu(a,a_1,...,a_k)]\mu(b_1,...,b_l).
\end{align*}

This shows the claim for all cylinder sets as  wanted.

\end{proof}

\section{$\mu$ it is not a  Gibbs probability for a continuous potential}
We continue to suppose that $H$ and $L$ satisfy the Assumption A.

\medskip

%Consider a generic element $x=(k_0,k_1,k_2,...,k_n,...)\,\in\,\{1,2\}^\mathbb{N}.$
We denote
$$a(k_0,k_1)= \frac{1}{2} \bigg(1-\frac{\beta_{k_0}}{\beta_{k_1}}\bigg),$$
$$ b(k_0,k_1)= \frac{1}{4} \beta_{k_0}\bigg(\frac{1}{\beta_{k_{1}}}-\beta_{k_{1}}\bigg),$$
and, $$\gamma=\frac{1}{4} \bigg(1-\beta_1^2\bigg) = \mu(1,1).$$
where $k_0,k_1\in \{1,2\}$, $\beta_1=\sin(2\theta)$ and $\beta_2=-\beta_1$.
The possible values of $a(k_0,k_1)$ and $b(k_0,k_1)$ are:\newline
a) if $k_0=k_1$, then $a(k_0,k_1)=0$ and $0< b(k_0,k_1)=\frac{1}{4} (1-\beta_1^2)=\gamma<\frac{1}{4};$\newline
b) if $k_0\neq k_1$, then $a(k_0,k_1)=1$ and $-\frac{1}{4}<b(k_0,k_1)=\frac{1}{4} (\beta_1^2-1)=-\gamma<0.$

\bigskip

From Proposition \ref{recursive},
\begin{equation}\label{eqrec}
\mu (k_0 , k_{1} , ...,k_n)=a(k_0,k_1)\mu (k_1 , ...,k_{n}) + b(k_0,k_1)\mu (k_2,...,k_{n}).
\end{equation}

\bigskip

\begin{proposition} Suppose  $\theta \neq \pi/4$. Then, $\mu$ is positive on cylinders sets.
\end{proposition}

\begin{proof} It can be checked directly that any cylinder of size 1, 2 or 3 has a positive probability. Suppose that $\mu(x_1,...,x_n)>0$ for any cylinder set of size $n$. As $\mu$ is stationary
\[\mu(x_0,...,x_n) = \mu(1,x_0,x_1,...x_n) + \mu(2,x_0,x_1,...,x_n) \geq \mu(x_0,x_0,x_1,...x_n)\]
\[ = a(x_0,x_0)\mu(x_0,...,x_n) + b(x_0,x_0)\mu(x_1,...,x_n) = \gamma\mu(x_1,...,x_n) >0.\]
Then, the result follows easily by induction. \end{proof}

\bigskip

%By the Shannon-McMillan-Breiman Theorem the limit
%\begin{equation} \label{SMB} -\lim_{n \to \infty} \frac{1}{n}\, \log \mu (k_0 , k_{1} , k_2, %k_3,...,k_n)=  \log J(\alpha)
%\end{equation}
%exist $\mu$-almost everywhere $\alpha=(k_0 , k_{1} , k_2, k_3,...,k_n,....)\in \{1,2\}^\mathbb{N}$ (see \cite{PolY}).

%Moreover $h(\mu) =\int \log J d \mu$ by Brin-Katok formula (see for instance \cite{Man}).

We get also a corollary about the Jacobian $J$ (which formally could be more properly  called the inverse Jacobian - particular case of Lebesgue measure) of $\mu$.
Define for $x=(x_0,x_1,x_2,...)$,
$$J_{\mu}(x)=J(x) := \lim_{n\to\infty} \frac{\mu(x_0,...,x_n)}{\mu(x_1,...,x_n)},$$
if, the limit exists. It is known that for any invariant measure $\nu$, the Jacobian $J_{\nu}$ is well defined for $\nu$ a.e. $x$ and can be seen as the Radon-Nikodym derivative of $\nu$ over the inverse branches of $\sigma$ (see \cite{PP}, \cite{Man}, \cite{Lo3} or \cite{PoYu}).

Moreover, the entropy satisfies $ h(\mu)= - \int \log J_{\mu} d \mu. $

We say that $f$ is normalized if for any $x\in \Omega$ we have $\sum_{\sigma(y)=x} e^{f(y)}=1$ (see \cite{PP}).

If $\nu$ is the equilibrium measure of a Lipschitz normalized function $f$, then $ J_{\nu}=e^f$ (see \cite{PP}).
\medskip

From (\ref{eqrec}), for any $n\geq 1$,
\begin{equation} \label{hhg}\frac{\mu (k_0 , k_{1} , ...,k_n)}{\mu(k_1,...,k_n)}=a(k_0,k_1) + b(k_0,k_1)\frac{\mu (k_2,...,k_{n})}{\mu(k_1,...,k_n)},
\end{equation}
then, we get the following:

\begin{corollary}
\begin{equation} \label{cfg} J(k_0,k_1,k_2,...) = a(k_0,k_1) + b(k_0,k_1)\frac{1}{J(k_1,k_2,k_3,...)}.
\end{equation}
\end{corollary}

\medskip

%The bottom line is

%\begin{equation} \label{ok}
%1\geq  J(k_0,k_1,k_2,...)= a(k_0,k_1) +  b(k_0,k_1)\,\frac{1}{J(k_1,k_2,k_3,...)}.
%\end{equation}
\medskip

\noindent
{\bf Remark:}  We have that $\gamma\leq J\leq 1-\gamma $.

Indeed,  denote by $J^n(x):=\frac{\mu(x_0,...,x_n)}{\mu(x_1,...,x_n)}$, $x=(x_0,x_1,...)$.
From (\ref{hhg}) we get
\[J^n(x_0,x_1,x_2,...) = a(x_0,x_1) + b(x_0,x_1)\frac{1}{J^{n-1}(x_1,x_2,x_3,...)}.\]
It follows that
\begin{equation}\label{desjac} 1 \geq J^{n+1}(x_0,x_0,x_1,x_2,x_3,...) = \gamma\frac{1}{J^{n}(x_0,x_1,x_2,x_3,...)},
\end{equation}
and, then
\[J^n(x_0,x_1,...) \geq \gamma ,\]
for any $n$ and $x=x_0,x_1,...\in \Omega$.

As the sum $J^n(1,x_2,x_3,..) + J^n(2,x_2,x_3,..)=1$ we get  also that $J^n\leq 1-\gamma$.
Finally, we observe that $J(x)=\lim_{n}J^n(x)$, if the limit exists.
%Note that
%$\mu$ satisfies  for all
%$(k_0 , k_{1} , ...,k_n).$

%\begin{equation} \label{kkk3}\frac{\mu (k_0 , k_{1} , ...,k_n)}{\mu(k_1,k_2,...,k_n)}=a %(k_0,k_1) + b(k_0,k_1)\,\frac{\mu (k_2,...,k_{n})}{\mu(k_1,k_2,...,k_n)}.
%\end{equation}

\medskip

\bigskip

\begin{proposition} \label{oui} The limit of  the sequence $\frac{\mu(x_0,...,x_n)}{\mu(x_1,...,x_n)}$, when $n \to \infty$, does not exists for $x=1^\infty = (1,1,1,1,1...)$.
\end{proposition}

\begin{proof} from (\ref{hhg}) we get

\begin{equation} \label{kkk}\frac{\mu (k_0 , k_{1} , ...,k_n)}{\mu(k_1,...,k_n)}=a (k_0,k_1) + b(k_0,k_1)\,\frac{\mu (k_2,...,k_{n})}{\mu(k_1,...,k_n)},
\end{equation}
where $a(1,1)=0 $ and $0<b(1,1)=\gamma<1$.
Then
$$\frac{\mu (\underbrace{1, ...,1}_{n+1})}{\mu(\underbrace{1,..,1}_n)}=\gamma\,\,\frac{\mu (\underbrace{1, ...,1}_{n-1})}{\mu(\underbrace{1, ...,1}_{n})}.$$

As $ \frac{\mu(1,1)}{ \mu(1)}= 2\, \gamma$
we get, by induction,
$$\frac{\mu (\underbrace{1, ...,1}_{n+1})}{\mu(\underbrace{1,..,1}_n)}=1/2,$$
for $n$ even, and
$$\frac{\mu (\underbrace{1, ...,1}_{n+1})}{\mu(\underbrace{1,..,1}_n)}=2\,\gamma,$$
for $n$ odd.
Therefore, $J(1^\infty)$ does not exist.
\end{proof}

\bigskip

%One can consider  on $\Omega=\{1,2\}^\mathbb{N} $ the metric $d(x,y)= \alpha^{-n}$, where $n$ is the first symbol of $x=(x_0,x_1,..)$ and $y=(y_0,y_1,..)$ such that they disagree, and $\alpha>1$.
Let $h(\mu)$ be the Kolmogorov-Sinai entropy of $\mu$.
We will use the following definition of Gibbs state:

\begin{definition} \label{ees} Suppose $f$ is a continuous potential $f:\{1,2\}^{\mathbb{N}} \to \mathbb{R}$, such that, for all $x\in \{0,1\}^{\mathbb{N}}$ we have
$\sum_{\sigma(y)=x} e^{f(y)}=1$. Then,  if $\nu$ is a $\sigma$-invariant probability such that:

$$\sup_{\rho} \,\bigg\{ h(\rho) + \int f d \rho \,\bigg\}= h(\nu) + \int f\, d\nu,$$
we say that $\nu$ is a Gibbs state for $f$.

\end{definition}

\begin{theorem} \label{qq} $\mu$ is not a Gibbs state for a continuous potential.

\end{theorem}

\begin{proof}
	Suppose by contradiction that $\mu$ is Gibbs for a continuous function $f$.
	Denote by $\mathcal{L}_f$ the Ruelle operator  for $f$ (see \cite{PP}). Note that $\mathcal{L}_f(1)=1$.
	It follows that exists a probability $\gamma$ that is a fixed point of the dual operator $\mathcal{L}_f^*,$
	that is, for any continuous function $g:\{1,2\}^{\mathbb{N}}\to\mathbb{R}$,
	\[ \int \sum_{i}e^{f(i,x_1,x_2,...)}g(i,x_1,x_2,...) \,d\gamma(x)= \int g(x)\,d\gamma(x). \]
	This shows that $e^f$ is almost everywhere equal to the Jacobian of $\gamma$ ($J_{\gamma}=e^f$ a.e.), because it is the  Radon-Nikodym derivative of $\gamma$ in the inverse branches. Following the arguments of Proposition 3.4 in \cite{PP} we get
	$$\sup_{\rho} \bigg\{\, h(\rho) + \int f d \rho\bigg\}= h(\gamma) + \int f \, d \gamma =0.$$
	By hypothesis we conclude that 	$h(\mu) + \int f d \mu=h(\gamma) + \int f \, d \gamma = 0.$ Following the arguments of Proposition 3.4 in \cite{PP} again, we get that $\mu$ is a fixed point of the dual operator $\mathcal{L}_f^*$.
	This means that for any function $\varphi$ we have
	$$ \int \mathcal{L}_f \, (\varphi)\, d \mu= \int \varphi d \mu.$$

	Taking $\varphi= I_{ (i_0,i_1,...,i_n)}$ we get that  $\mathcal{L}_f \, (\varphi)(x)=  I_{ (i_1,i_2,...,i_n)} (x)\,e^{f ( i_0 x)}$. Therefore,
	$$ \mu (i_0,i_1,...,i_n) = \int_{ (i_1,i_2,...,i_n)} \,e^{f ( i_0 x)}\, d \mu(x),$$
	and, hence
	$$ \frac{\mu (i_0,i_1,...,i_n) }{\mu (i_1,i_2,...,i_n) } =\frac{ \int_{ (i_1,i_2,...,i_n)} \,e^{f ( i_0 x)}\, d \mu(x)}{\mu (i_1,i_2,...,i_n) } .$$

	The right hand side is an average of values of $e^f$  over decreasing cylinder sets.
	Since 	$e^{f}$ is assumed to be a continuous function, it converges uniformly
	in the sequence $(i_0,i_1,...,i_n..)$ to $e^{f ( i_0,i_1,...,i_n.. )}.$
	Then, for any $i\in\{1,2\}$, we have that for any $x= (i,x_1,...,x_n,...)$ the limit,
	$$\lim_{n\to\infty} \frac{\mu(i,x_1,...,x_n)}{\mu(x_1,...,x_n)}= e^{ f(i,x_1,...,x_n,...)},$$
exists. 	This is a contradiction because, from Proposition \ref{oui}, this property is not true for $x=(1^\infty)$.
	
\end{proof}

We will make some final remarks about $J$.
According to (\ref{cfg}),
$J:\{1,2\}^\mathbb{N} \to \mathbb{R}$ is an unknown  function which satisfies the property:

$$J(k_0,k_1,k_2,...) =  a(k_0,k_1) +  b(k_0,k_1)\,\frac{1}{J(k_1,k_2,k_3,...)}$$
$$ =a(k_0,k_1) +  b(k_0,k_1)\,\frac{1}{a(k_1,k_2) +  b(k_1,k_2)\,\frac{1}{J(k_2,k_3,k_4,...)}}.$$

\medskip
Note that
$$\frac{\mu(k_0,...,k_n)}{\mu(k_1,...,k_n)}=a(k_0,k_1) +  b(k_0,k_1)\,\frac{1}{a(k_1,k_2) +  b(k_1,k_2)\frac{1}{...a(k_{n-1},k_n)+b(k_{n-1},k_n)\frac{1}{1/2}}},$$
then, we obtain the following result:

\begin{lemma}\label{jexpansion}
\begin{align*}
J(k_0&,k_1,k_2,...) = \lim_{n} \frac{\mu(k_0,...,k_n)}{\mu(k_1,...,k_n)}\\
&=  \lim_{n} \left[a(k_0,k_1) +  b(k_0,k_1)\,\frac{1}{a(k_1,k_2) +  b(k_1,k_2)\frac{1}{...a(k_{n-1},k_n)+b(k_{n-1},k_n)\frac{1}{1/2}}}\right],
\end{align*}
if the limit exists.
\end{lemma} In this sense ${J}$ has an expression in continued fraction (see expression (b) in page 2 in \cite{Wall} for the general setting).

We want to show that (when exists) $J$ assume only two possible values. In the proof, the following lemma will be used.

\begin{lemma}\label{1221}
	For any $n\geq 1$ and $k_1,...,k_n \in \{1,2\}$ we have
	\[ \frac{\mu(1,2,2,1,1,k_1,...,k_n)}{\mu(2,2,1,1,k_1,...,k_n)}=\frac{\mu(1,k_1,...,k_n)}{\mu(k_1,...,k_n)}.\]
	Particularly,  $J$ is not defined in $(1,2,2,1)^\infty=(1,2,2,1,1,2,2,1,...)$	
\end{lemma}
\begin{proof}
Note that
\[\frac{\mu(1,2,2,1,1,k_1,...,k_n)}{\mu(2,2,1,1,k_1,...,k_n)}=1-\gamma\frac{1}{0+\gamma\frac{1}{1-\gamma\frac{1}{0+\gamma \frac{1}{\frac{\mu(1,k_1,...,k_n)}{\mu(k_1,...,k_n)}}  }}}= \frac{\mu(1,k_1,...,k_n)}{\mu(k_1,...,k_n)}.\]
Particularly,
\[\frac{\mu(1,2,2,1,...,1,2,2,1,1,2)}{\mu(2,2,1,...,1,2,2,1,1,2)}=\frac{\mu(1,2)}{\mu(2)}=\frac{1}{2}(1+\beta_1^2),\]
while
\[\frac{\mu(1,2,2,1,...,1,2,2,1,1)}{\mu(2,2,1,...,1,2,2,1,1)}=\frac{1}{2}.\]
This proves that $J$ is not defined in $(1,2,2,1)^\infty$.

\end{proof}

%\begin{definition}
%Given $(k_0,k_1,k_3,...,k_n,...)$  we denote $\tilde{J}$ the continuous fraction expansion expression

%$$\tilde{J}(k_0,k_1,k_3,...,k_n,...)\,=$$
%\begin{equation} \label{yy1}\, a(k_0,k_1) +  b(k_0,k_1)\,\frac{1}{a(k_1,k_2) +  b(k_1,k_2)\,\frac{1}{a(k_2,k_3) +  b(k_2,k_3)\frac{1}{...}}},
%\end{equation}
%if this converges,  that is, if there exist the limit of the finite truncations
%$$\lim_{n} a(k_0,k_1) +  b(k_0,k_1)\,\frac{1}{a(k_1,k_2) +  b(k_1,k_2)\frac{1}{...a(k_{n-1},k_n)+b(k_{n-1},k_n)\frac{1}{1/2}}}.$$

%\end{definition}

\begin{proposition}
The only possible (convergent) values of $J$ are $p=\frac{1+\beta_1}2$, or  $1-p=\frac{1-\beta_1}2$.	
\end{proposition}

\begin{proof}
Remember that

a) if $k_0=k_1$, then $a(k_0,k_1)=0$ and $b(k_0,k_1)=\gamma=\frac{1}{4} (1-\beta_1^2)$

b) if $k_0\neq k_1$, then $a(k_0,k_1)=1$ and $b(k_0,k_1)=-\gamma=\frac{1}{4} (\beta_1^2-1)$.

\noindent
In this way the value of $J$ does not change if we permute $1$ and $2$ in the sequence $(k_0,k_1,...)\in \{1,2\}^\mathbb{N}$. For example
\[J(1,1,1,2,2,1,...)=J(2,2,2,1,1,2,...)\]
(if the limit exists).
Therefore, we will introduce another code.
%If $k_0=k_1$ then we associated  the symbol $a$ and  if $k_0\neq k_1$ we associate the symbol $b$.
For each given sequence $k\in \{1,2\}^\mathbb{N}$, $k=(k_0,k_1,k_2,...)$ we associate a new sequence $m=m(k)=(m_0,m_1,m_2,..)\in\{a,b\}^\mathbb{N}$ by the rule: $m_i=a$ if $k_i=k_{i+1}$ and $m_i=b$ if $k_i\neq k_{i+1}$. So, we are looking if there is a change, or not, in the string $k$ by
using the rules
$ \underbrace{11}_a$\,, \,$ \underbrace{22}_a$\,,\,$ \underbrace{12}_b$\,,\, $ \underbrace{21}_b$.

For example, given a sequence $k$ of the form
$$ k=(1,2,1,1,2,2,... ),$$
then, we associate
$m=(b,b,a,b,..).$ Clearly, we can consider $J$ defined over  $\{a,b\}^\mathbb{N}$, from $J(m(k)):=J(k)$.
\bigskip

%Given $m$ we can find  the $k=(k_0,k_1,k_2,...)$  to whom $m$ it is associated if we know the first element $k_0=1$ or $k_0=2$.

\medskip

It can be checked that:
\begin{align*}
J(a,a,m_3,..) &= J(m_3,m_4,..)\\
J(a,b,m_3,..) &= \frac{1}{ \gamma^{-1}- \frac{1}{   J(m_3,m_4,..)}}\\
J(b,a,m_3,..) &= 1- J(m_3,m_4,..)\\
\text{and}\\
J(b,b,m_3,..) &= 1 - \frac{1}{ \gamma^{-1}- \frac{1}{   J(m_3,m_4,..)}}.
\end{align*}

\medskip

%For example, in terms of sequences $k$ the item

%a) If $m_1=a,m_2=a$, then $J(m_1,m_2,m_3,..) = J(m_3,m_4,..)$ means
%\medskip

%I) $J(1,1,1,k_3,..) = J(1,k_3,..)$ and

%II) $J(2,2,2,k_3,..) = J(2,k_3,..)$

%\medskip

%In this way if $k=(a_0,a_2,a_3,..) \in \{1,2\}^\mathbb{N}$ is an element where the fraction expansion $J(k)$ limit exists, then this value
%does not change if we delete $00$ (respectively, $11$) of the sequence $k$ where appears $000$ (respectively, $111$). Indeed, this corresponds to $m_1=a=m_2$. In other words, this corresponds to a sequence $a\,a$.
In this way when looking to sequences $m$,  the finite strings with $(aa)^{n}$, $n\in \mathbb{N}$, can be deleted (when $J$ converges). That is,
\[J(m_0,m_1,...,m_j,a,a,m_{j+2},...)=J(m_0,m_1,...,m_j,m_{j+2},...). \]
Furthermore, it can be checked that
\[J(a,b,a,b,m_5,m_6,...)=J(b,a,b,a,m_5,m_6,...)=J(m_5,m_6,m_7,...).\]
In this way, if the fraction expansion $J(m)$  exists, then one can show that this value
does not change if we delete finite parts of the sequence where appears  $baba$ or $abab$.

In a  first moment we consider only the possibility of deleting or including the string $aa$ in the sequence.
From Proposition \ref{oui} %and lemma \ref{1221}
we obtain that if $m=(m_1,...,m_k,a,a,a...)$, then $J(m)$ is not defined. For the other cases we consider the letters of $m$ as block of length 2

$$m=([m_1,m_2],[m_3,m_4],[m_5,m_6],...).$$

As we are interested in the value of $J$, we can assume that no blocks have the form $[a,a]$ (we can delete them), and, also that no blocks have the form $[b,b]$, because we can change this one for the pair of blocks $[b,a],[a,b]$. %The same occur if $m =(m_1,...,m_k,(abab)^{\infty})$ or $m=(m_1,...,m_k,(baba)^{\infty})$.
%This means  when considering $abab$ we get for example:
%$$ J(1,1,2,2,1,k_5..)=J(1,k_5..),$$
%that is we can delete $1,1,2,2$.
%For example
%$$ J(a,b,a,b,m_5,m_6,..) = J(m_5,m_6,..),$$
%$$ J(a,b,a,b,a,b,a,b,m_9,m_{10},..) = J(m_9,m_{10},..),$$
%$$ J(b,a,b,a,m_5,m_6,..) = J(m_5,m_6,..),$$
%and
%$$ J(m_1,m_2,a,b,a,b,m_7,m_8,..) = J( m_1,m_2,m_7,m_8,..).$$

From now on it is natural to consider one level up of symbolic representation. We get a new code introducing a new dictionary where we associate $\alpha=ab$ and $\beta=ba$. %As we have seen, we can substitute $aa$ by $\emptyset$.
%One can also substitute
%$$bb  \,\,\,\text{by} \, \,\, \beta\,\alpha= baab.$$
In this way, for $m=(m_0,m_1,m_2,...)=([m_0,m_1],[m_2,m_3],...)$ we associate $w=(w_0,w_1,w_2,...)$, where $w_i=\alpha$, if $[m_{2i},m_{2i+1}]=[a,b]$, and $w_i=\beta$, if $[m_{2i},m_{2i+1}]=[b,a]$.
%So we can divide an $m$ string in blocks of two letters and make the above substitutions, obtaining
%a new string $w=(w_1,w_2,w_3,...)\in \{\alpha,\beta\}^\mathbb{N}.$

%For example we associate

%$$m =([a,b],[b,a],[b,a],[a,b],...)$$ to
%$$ w=( \alpha,\beta,\beta,\alpha,...),$$

%$$m =(a,b,a,b,a,b,b,a,...)$$ to
%$$ w=( \alpha,\alpha,\alpha,\beta,...),$$
%and

%$$m =([a,b],[a,a],[b,b],[a,b],...)$$ to
%$$ w=( \alpha,\beta,\alpha,\alpha,...).$$
We need to study the possible values of $J$ over $\{\alpha,\beta\}^{\mathbb{N}}$.
First we remark that the string $\alpha,\alpha$ in $w$ corresponds to the string $a,b,a,b$ in $m$, which can be deleted without changes of the value of $J$. From Lemma \ref{1221} we get that $J$ is not defined if $w=(w_1,...,w_k,\alpha,\alpha,\alpha,...)$. For any other $w$ there is some sequence $w'$ with the same value for $J$ and such that $w'$ does not contain $\alpha,\alpha$. Furthermore, the same argument can be applied for the string $\beta,\beta$ in $w$. The conclusion is that for any sequence where $J$ is well defined, it is equal to $J(\alpha,\beta,\alpha,\beta,\alpha,...)$ or $J(\beta,\alpha,\beta,\alpha,\beta,...)$.

%
%Then we can suppose that  By considering the compositions of the corresponding functions, one can substitute (delete, indeed) parts of $w$ strings
%

%$$\beta\,\beta= baba  \,\,\,\text{by } \,\emptyset $$

%$$ \alpha\,\alpha=abab  \,\,\,\text{by } \,\emptyset. $$

%We present some examples in order to illustrate the code:  from the rules about $baba$ and $abab$ we get that in the case $J$ converges
%$$J( \alpha,\alpha,\alpha,\beta,w_5...)= J(\alpha,\beta,w_5,...),$$

%$$J( \alpha,\alpha,\alpha,\alpha,\beta,w_6,...)= J(\beta,w_6,...),$$

%$$J( \beta,\beta,\beta,\beta,\alpha,w_6,...)= J(\alpha,w_6,...)$$
%and
%$$J(\beta,\alpha, \alpha,\alpha,\alpha,\alpha,\beta,w_8,...)= J(\beta,\alpha, \beta, w_8,...).$$

%In the finite fraction expansion of odd order of $k$ (which is associated to a certain string $m$ of even order and so to a string $w$) we can delete parts (in the $\alpha,\beta$ dictionary expansion)
%in such a way that we end up with the estimation of $J$ in a istring $w$ of one of the kinds:

%a)$(\alpha\,\beta)^n$ or $(\alpha\,\beta)^n\alpha$,

%or

%b) $(\beta\,\alpha)^n$ or $(\beta\,\alpha)^n\beta$,

%\medskip
%Indeed, note that
%$$(\beta\,\alpha)^{n_1}\beta  (\alpha\,\beta)^{n_2} \,\alpha\,(\beta\,\alpha)^{n_3}\, \beta\, \, (\beta\,\alpha)^{n_4}\,\beta\,(\alpha\,\beta)^{n_5} \alpha (\beta\,\alpha)^{n_6}  ...(\beta\,\alpha)^{n_{s-1}}\beta  (\alpha\,\beta)^{n_s}=$$
%$$(\beta\,\alpha)^{n_1+n_2+\dots+n_s+s/2-1}\beta$$
%\medskip
From now on we are interested in to determine what are the possible two values of $J$. Consider the transformation

$$ x\,\to\,f_1(x)= 1- \frac{1}{\gamma^{-1} - \frac {1}{x}}.$$
The string $\beta \alpha$ means $baab$, which corresponds to
$J(m_1,m_2,m_3,..) = 1 - \frac{1}{ \gamma^{-1}- \frac{1}{   J(m_3,m_4,..)}..}.$
Note that if the expansion $J(m_3,m_4,..)$ exists for the string $(m_3,m_4,..)$, then it also exists the one for $J(m_1,m_2,m_3,..) $.
In this way $f_1(J (m_3,m_4,..))= J(m_1,m_2,m_3,..) .$

\medskip

Consider now $f_2$ defined by

$$ x \,\to\, f_2(x)=\frac{1}{\gamma^{-1}-\frac {1}{1-x}}\,\,.$$
The string $ \alpha\, \beta$ means $abba$, that is, it corresponds to
$J(m_1,m_2,m_3,..) = \frac{1}{\gamma^{-1}-\frac {1}{1-J(m_3,m_4,..)}}.$
In this way $f_2(J(m_3,m_4,..))= J(m_1,m_2,m_3,..).$

\medskip

Note that the fixed points for both functions $f_1(x)= 1- \frac{1}{\gamma^{-1} - \frac {1}{x}}$ and
$f_2(x)=\frac{1}{\gamma^{-1}-\frac {1}{1-x}}$ are the same:  $p=\frac{1+\beta_1}{2}$ and $1-p=\frac{1-\beta_1}{2}$.
Furthermore, the interval $[1-p,p]$ is invariant by  $f_1$ and also by $f_2$. The point $p$ is a global attractor for $f_1$ in $(1-p,p]$, and, $1-p$ is a global attractor for $f_2$ in $[1-p,p)$.

As we have seen in Lemma \ref{jexpansion} it is natural to truncate $J(k_0,k_1,k_2,\dots)$ (at level $r$ for instance) by taking in the last position $r$, in the expansion of $J$, the value $1/2$. As $1/2  \in [1-p,p]$ (in fact is its center), and the interval $[1-p,p]$ is left invariant by the diffeomorphisms $g_{a}(x)=\frac{\gamma}{x}$ and $g_{b}(x)=1-\frac{\gamma}{x}$, when the limit exists the successive truncations should converge to $p$ or to $1-p$.

Then, the only possible (convergent) values attained by the continuous fraction expansion of $J$ in Lemma \ref{jexpansion} are $p$ or  $1-p$.
\medskip

\end{proof}

\medskip

From the above Proposition one can estimate the entropy of $\mu$ via the expression $ h(\mu)= - \int \log J_{\mu} d \mu. $

\section{A Large Deviation Principle for $\mu$} \label{LDPs}

We assume that $H$ and $L$ satisfy the Assumption A.

We want to prove the Theorem \ref{teo2}. In this way denote
$$ c (t) = \lim_{n \to \infty}  \frac{1}{n}\,\, \log \int e^{t\,(\,A(z) + A(\sigma(z)) + A(\sigma^{2} (z)) + ...
	+ A(\sigma^{n-1} (z))\, \,)} d \mu (z),$$
for each $t \in \mathbb{R}$.

The function
$c(t)$ is  called the free energy on the point $t$ for the probability $\mu$ and the classical observable $A$ (see \cite{El}). As we said before it is a classical result that if $c$ is differentiable then a Large Deviation Principle is true for $\mu$ and $A$ (see \cite{Lo3} for a proof with details). In this case the deviation function $I$ is the Legendre transform of $c$ (see \cite{El} or \cite{Lo3}).

We will show the explicit expression of $c(t)$ which is clearly a differentiable function - therefore, will follow the LDP claim  of Theorem \ref{teo2}. More precisely, we will prove the following result.

\begin{proposition}\label{c} If $A:\{1,2\}^{\mathbb{N}}\to\mathbb{R}$ depends only on the first coordinate of $x$, then for any $t\in \mathbb{R}$,
	\begin{equation*} c(t)= \frac{1}{2}\log \bigg(\Big[ e^{t\, A(1)}+e^{t\,A(2)}\Big]^{2}-\Big[ \beta_{1}e^{t\, A(1)}+\beta_2 e^{t\, A(2)}\Big]^{2}\bigg)-\log(2).
	\end{equation*}
\end{proposition}

\bigskip

Note that the explicit expression of $c(t)$ permits to estimate explicitly (via Legendre transform) the deviation function $I$ (and, so the $M$ above).
\medskip

Suppose $A:\{1,2\} \to \mathbb{R}$ is a function. We define, for $t\in \mathbb{R}$,

$$Q_n(t) = \int e^{t\, (  A(x_0) + A(x_1) +...+ A(x_n) ) }d \mu (x) = $$
$$ \sum_{j_0}\, \sum_{j_1}\, ...\sum_{j_n} e^{t\, (  A(j_0) + A(j_1) +...+ A(j_n) ) }\, \mu ( j_0,j_1,...j_n).$$

Denote $\displaystyle\alpha(t)=  \sum_{j} \beta_{j}\, e^{t\, \, A(j)}$ and $\displaystyle\delta(t)= \sum_{j}  e^{t\, \, A(j)}$. Note that $|\alpha(t)|< | \delta(t)|$.

Before studying the general case of $Q_n$ we will consider an example.
\begin{example}
	We will compute $Q_3(t)$. From example \ref{ex},

	$$\mu(j_0,j_1,j_2,j_3) = \frac{\mu(j_1,j_2,j_3)}{2} - \frac{\beta_{j_0} \beta_{j_1}}{4}\mu(j_2,j_3)   +  \frac{\beta_{j_0} \beta_{j_2}}{16}-  \frac{\beta_{j_0} \beta_{j_3}}{16},$$
	then, we get
	$$Q_3(t) = \sum_{j_0}\, \sum_{j_1}\, \sum_{j_2}  \sum_{j_3} e^{t\, (  A(j_0) + A(j_1) +A(j_2) + A(j_3)) }\, \mu ( j_0,j_1,j_2,j_3)=$$
	$$ =\bigg[\frac{1}{2}\sum_{j_0}\, e^{t\,   A(j_0)  }\,\bigg]\,\sum_{j_1}\, \sum_{j_2} \sum_{j_3}\, e^{t\, (   A(j_1) +A(j_2)+ A(j_3)) }\, \mu (j_1,j_2,j_3)\,+ $$
	$$
	-\frac{1}{4}\bigg[\sum_{j_0}e^{t\, A(j_0)}\beta_{j_0}\bigg]\, \bigg[\sum_{j_1}e^{t\,A(j_1)}\beta_{j_1}\bigg]\, \sum_{j_2,j_3}   e^{t\, ( A(j_2)+ A(j_3)) } \mu(j_2,j_3)
	$$
	$$+\frac{1}{16}\sum_{j_0}\, \sum_{j_1}\, \sum_{j_2}   \sum_{j_3} e^{t\, (  A(j_0) + A(j_1) +A(j_2))+ A(j_3) }\beta_{j_0} \beta_{j_{2}}$$
	$$-\frac{1}{16}\sum_{j_0}\, \sum_{j_1}\, \sum_{j_2}   \sum_{j_3} e^{t\, (  A(j_0) + A(j_1) +A(j_2))+ A(j_3)) }\beta_{j_0} \beta_{j_{3}}=$$
	$$ =\frac{1}{2}\delta(t)\,\ Q_2(t) -  \frac{1}{4}\alpha(t)^2 Q_1(t).
	$$
	
\end{example}

Now we will extend the above ideas for  a more general result.

\begin{theorem}\label{teoQ} For any $n\in \mathbb{N}$ and $t\in\mathbb{R}$, we have
	
	\begin{equation}\label{main1}
	Q_n (t)=\left[\begin{array}{l}\frac{1}{2}\delta(t)\,Q_{n-1}(t)- \frac{1}{4}\alpha(t)^2\,Q_{n-2}(t)+ \frac{1}{8}\delta(t)\,\alpha(t)^2\, Q_{n-3}(t) \\ \\
	\,\,-  \frac{1}{16}\delta(t)^2\,\alpha(t)^2\, \, Q_{n-4}(t)+\frac{1}{32}\delta(t)^3\,\alpha(t)^2\, \, Q_{n-5}(t)-...+\\  \\      +\frac{(-1)^{n-1}}{2^{n-2}}\, \delta(t)^{n-4}\,\alpha(t)^2\, \, Q_{ 2}(t)
	+ \frac{(-1)^n}{2^{n-1}}\, \delta(t)^{n-3}\,\alpha(t)^2\, \, Q_{ 1}(t)\end{array}\right] .
	\end{equation}
	
\end{theorem}

\begin{proof}
	By definition
	$$Q_n(t) =  \sum_{j_0}\, \sum_{j_1}\, ...\sum_{j_n} e^{t\, (  A(j_0) + A(j_1) +...+ A(j_n) ) }\, \mu ( j_0,j_1,...j_n).$$
	From equation (\ref{eq0})
	
	$$   \mu (j_0 , j_1 , j_2,...,j_n) \,=\,\,\, \frac{\mu( j_1,j_2,...,j_n )}{2} - \frac{\beta_{j_0}\, \beta_{j_1}}{2^{2}} \, \mu (j_2,...,j_n) + \,$$
	$$
	\frac{\beta_{j_0}\beta_{j_2}}{2^{3}} \mu( j_3,j_4,...,j_n ) - \frac{\beta_{j_0} \beta_{j_3}}{2^{4}} \mu( j_4,j_5,...,j_n )+...$$
	$$  +\frac{(-1)^{n-2} \beta_{j_0} \beta_{j_{n-2}}}{2^{n-1}} \mu(j_{n-1},j_n)+ \frac{(-1)^{n-1} \beta_{j_0} \beta_{j_{n-1}}}{2^{n+1}}+ \frac{(-1)^{n} \beta_{j_0} \beta_{j_n}}{2^{n+1}} .
	$$
	Then, we get
	
	$$
	Q_n(t)=\frac{1}{2}\sum_{j_0}e^{t\,  A(j_0)}\sum_{j_1,...,j_n} e^{t\, ( A(j_1) +...+ A(j_n) ) }\, \mu (j_1,...j_n)-
	$$
	$$-\frac{1}{2^2}\sum_{j_0}e^{t\,  A(j_0)}\beta_{j_0}\sum_{j_1}e^{tA(j_1)}\beta_{j_1}\sum_{j_2,...,j_n} e^{t\, ( A(j_2) +...+ A(j_n) ) }\, \mu (j_2,...j_n)$$
	$$+\frac{1}{2^3}\sum_{j_0}e^{tA(j_0)}\beta_{j_0}\sum_{j_1}e^{tA(j_1)}\sum_{j_2}e^{tA(j_2)}\beta_{j_2}\sum_{j_3,...,j_n} e^{t\, ( A(j_3) +...+ A(j_n) ) }\, \mu (j_3,...j_n)$$
	$$...+ \frac{(-1)^{n}}{2^{n+1}}  \sum_{j_0}e^{tA(j_0)}\beta_{j_0} \sum_{j_n}e^{tA(j_n)}\beta_{j_n}\sum_{j_2}e^{tA(j_2)}...\sum_{j_{n-1}}e^{tA(j_{n-1})}$$
	(using the definitions of $\alpha(t)$ and $\delta(t)$)
	$$=\frac{1}{2}\delta(t)\,Q_{n-1}(t)- \frac{1}{4}\alpha(t)^2\,Q_{n-2}(t)+ \frac{1}{8}\delta(t)\,\alpha(t)^2\, Q_{n-3}(t)-  \frac{1}{16}\delta(t)^2\,\alpha(t)^2\, \, Q_{n-4}(t)+$$
	$$\frac{1}{32}\delta(t)^3\,\alpha(t)^2\, \, Q_{n-5}(t)-...      + \frac{(-1)^{n-1}}{2^{n-2}}\, \delta(t)^{n-4}\,\alpha(t)^2\, \, Q_{ 2}(t)  $$
	$$+ \frac{(-1)^n}{2^{n-1}}\, \delta(t)^{n-3}\,\alpha(t)^2\, \, Q_{ 1}(t)
	+
	\frac{(-1)^{n-1}}{2^{n+1}}\alpha(t)^2\delta(t)^{n-2}+ \frac{(-1)^{n}}{2^{n+1}}\alpha(t)^2\delta(t)^{n-2}
	$$
	
	$$=\frac{1}{2}\delta(t)\,Q_{n-1}(t)- \frac{1}{4}\alpha(t)^2\,Q_{n-2}(t)+ \frac{1}{8}\delta(t)\,\alpha(t)^2\, Q_{n-3}(t)-  \frac{1}{16}\delta(t)^2\,\alpha(t)^2\, \, Q_{n-4}(t)+$$
	$$\frac{1}{32}\delta(t)^3\,\alpha(t)^2\, \, Q_{n-5}(t)-...      + \frac{(-1)^{n-1}}{2^{n-2}}\, \delta(t)^{n-4}\,\alpha(t)^2\, \, Q_{ 2}(t) + \frac{(-1)^n}{2^{n-1}}\, \delta(t)^{n-3}\,\alpha(t)^2\, \, Q_{ 1}(t).
	$$
	
\end{proof}

\begin{proposition}
	\[Q_{n+2}(t) = \frac{\delta(t)^{2}-\alpha(t)^{2}}{4}Q_{n}(t).\]
\end{proposition}

\begin{proof}
	
	We have that
	$$ Q_n (t)\,=\frac{1}{2}\delta(t)\,Q_{n-1}(t)- \frac{1}{4}\alpha(t)^2\,Q_{n-2}(t)+ \frac{1}{8}\delta(t)\,\alpha(t)^2\, Q_{n-3}(t)-  \frac{1}{16}\delta(t)^2\,\alpha(t)^2\, \, Q_{n-4}(t)+$$
	$$\frac{1}{32}\delta(t)^3\,\alpha(t)^2\, \, Q_{n-5}(t)-...      + \frac{(-1)^{n-1}}{2^{n-2}}\, \delta(t)^{n-4}\,\alpha(t)^2\, \, Q_{ 2}(t)  + \frac{(-1)^n}{2^{n-1}}\, \delta(t)^{n-3}\,\alpha(t)^2\, \, Q_{ 1}(t).
	$$
	
	Using the same formula applied in $Q_{n-1}(t)$ we get
	$$Q_{n-1}(t) = \frac{1}{2}\delta(t)\,Q_{n-2}(t)- \frac{1}{4}\alpha(t)^2\,Q_{n-3}(t)+ \frac{1}{8}\delta(t)\,\alpha(t)^2\, Q_{n-4}(t)-  \frac{1}{16}\delta(t)^2\,\alpha(t)^2\, \, Q_{n-5}(t)+$$
	$$\frac{1}{32}\delta(t)^3\,\alpha(t)^2\, \, Q_{n-6}(t)-...      + \frac{(-1)^{n-2}}{2^{n-3}}\, \delta(t)^{n-5}\,\alpha(t)^2\, \, Q_{ 2}(t)  + \frac{(-1)^{n-1}}{2^{n-2}}\, \delta(t)^{n-4}\,\alpha(t)^2\, \, Q_{ 1}(t).
	$$
	If we replace $ Q_{n-1} $ in the first equation for the right hand side of the second equation we get:
	$$ Q_n (t)\,=\frac{\delta(t)^{2}-\alpha(t)^{2}}{4}\,Q_{n-2}(t). $$
\end{proof}

\bigskip

\noindent
{\bf Proof of Proposition \ref{c}:}
From the above result we get as a corollary that $Q_n(t)$ growths exponentially like    $e^{\frac{\sqrt{\delta(t)^{2}-\alpha(t)^{2}}}{2}}$.
In this way for each fixed $t$:
$$ c(t) = \lim_{n \to \infty} \frac{1}{n} \log Q_n (t) = \log \frac{\sqrt{\delta(t)^{2}-\alpha(t)^{2}}}{2}$$
\begin{equation} \label{oot} = \frac{1}{2}\log \bigg(\Big[\sum_{j_0}  e^{t\, \, A(j_0)}\Big]^{2}-\Big[\sum_{j_0}  \beta_{j_0}e^{t\, \, A(j_0)}\Big]^{2}\bigg)-\log(2).
\end{equation}
\qed

The above function $c$ is differentiable on $t$.
\medskip

\section{Appendix}

On the Appendix we will give the proof of several technical results which were used before.

\medskip

\begin{proposition} \label{stat1} For $n\in \{2,3,...\}$,
\[\mu_{\infty,n+1} (1 , j_1 , j_3,...,j_n)+ \mu_{\infty,n+1} (2, j_1 , j_3,..,j_n)= \mu_{\infty,n} ( j_1 , j_3,..,j_n).\]

\end{proposition}

\bigskip

\begin{proof} As $ P_1\otimes P_{j_1} \otimes  P_{j_2}\,\otimes...\otimes P_{j_n}+P_2\otimes P_{j_1} \otimes  P_{j_2}\,\otimes...\otimes P_{j_n}=I\otimes P_{j_1} \otimes  P_{j_2}\,\otimes...\otimes P_{j_n}$, we get
 $$\mu_{\infty,n+1}(1, j_1 , j_2,...,j_n)+\mu_{\infty,n+1}(2, j_1 , j_2,...,j_n)=$$
 $$ =\frac{1}{2^{n+1}}\text{Tr}\left[\prod_{i=1}^{n} [\, I^{\otimes n+1} \,-\, (\sigma^x_{i}\otimes \sigma^x_{i+1})_{n+1}]( I\otimes  P_{j_1} \otimes  P_{j_2}\,\otimes...\otimes P_{j_n})\right]$$
  $$= \frac{1}{2^{n+1}}\text{Tr}\left[I^{\otimes n+1} \circ \prod_{i=2}^{n} [\, I^{\otimes n+1} \,-\, (\sigma^x_{i}\otimes \sigma^x_{i+1})_{n+1}]( I\otimes  P_{j_1} \otimes  P_{j_2}\,\otimes...\otimes P_{j_n})\right],$$
where in the last equality we use the expression
\footnotesize
\[\text{Tr} \left[-(\sigma^x_{1}\otimes \sigma^x_{2})_{n+1} \circ \prod_{i=2}^{n} [\, I^{\otimes n+1} \,-\, (\sigma^x_{i}\otimes \sigma^x_{i+1})_{n+1}]( I\otimes  P_{j_1} \otimes  P_{j_2}\,\otimes...\otimes P_{j_n})\right]=0,   \]
\normalsize
which follows from Tr$(\sigma^x)=0$.

 Note also that, for any $B_1,...,B_n$, we have
$$\text{Tr} (I \otimes B_1
\otimes ...\otimes B_n)=2\, \text{Tr} (B_1
\otimes ...\otimes B_n).$$
Then, we get
$$\mu_{\infty,n+1}(1, j_1 , j_2,...,j_n)+\mu_{\infty,n+1}(2, j_1 , j_2,...,j_n)$$
$$= \frac{1}{2^{n+1}}\text{Tr}\left[I^{\otimes n+1} \circ \prod_{i=2}^{n} [\, I^{\otimes n+1} \,-\, (\sigma^x_{i}\otimes \sigma^x_{i+1})_{n+1}]( I\otimes  P_{j_1} \otimes  P_{j_2}\,\otimes...\otimes P_{j_n})\right]$$
$$= \frac{1}{2^{n}}\text{Tr}\left[\prod_{i=1}^{n-1} [\, I^{\otimes n} \,-\, (\sigma^x_{i}\otimes \sigma^x_{i+1})_{n}](   P_{j_1} \otimes  P_{j_2}\,\otimes...\otimes P_{j_n})\right]$$
$$=\mu_{\infty,n}( j_1 , j_2,...,j_n).$$
\end{proof}

\medskip

\begin{theorem}\label{medidacilindro1}
Suppose $H$ and $L$ satisfy the Assumption A. Let $\mu$ be the probability measure defined in Corollary \ref{sta}. For any $n\geq 2$, we have
\small
\begin{equation}\label{eq0}
\mu(k,j_1,...,j_n)= \frac{\mu( j_1,...,j_n )}{2}  + \sum_{i=1}^{n-2}\frac{(-1)^i \beta_k \beta_{j_i} }{2^{i+1}}\mu(j_{i+1},...,j_n)  +\frac{(-1)^n\beta_k(\beta_{j_n}-\beta_{j_{n-1}})}{2^{n+1}} .
\end{equation}
\normalsize
and
\[\mu(k,j_1)= \frac{\mu(j_1)}{2} -\frac{\beta_k\beta_{j_1}}{4}.\]
\end{theorem}

\begin{proof}
For $n=1$ the above result can be verified directly, because
\[\mu(k,j_1)=\frac{1-\beta_k\beta_{j_1}}{4},\,\,\,\text{and}\,\,\, \mu(j_1)=\frac{1}{2}.\]

For $n\geq 2$, we have

\[\mu(k, {j_1} ,  {j_2},..., {j_n}  )
=\frac{1}{2^{n+1}}\text{Tr}\left[\prod_{i=1}^{n} [\, I^{\otimes n+1} \,-\, (\sigma^x_{i}\otimes \sigma^x_{i+1})_{n+1}](   P_k\otimes P_{j_1} \otimes  P_{j_2}\,\otimes...\otimes P_{j_n})\right]\]
\[= \frac{1}{2^{n+1}}\text{Tr}\left[I^{\otimes n+1}\circ \prod_{i=2}^{n} [\, I^{\otimes n+1} \,-\, (\sigma^x_{i}\otimes \sigma^x_{i+1})_{n+1}](   P_k\otimes P_{j_1} \otimes  P_{j_2}\,\otimes...\otimes P_{j_n})\right]          \]
\[\,\,\,\,- \frac{1}{2^{n+1}}\text{Tr}\left[(\sigma^x_{1}\otimes \sigma^x_{2})_{n+1}\circ\prod_{i=2}^{n} [\, I^{\otimes n+1} \,-\, (\sigma^x_{i}\otimes \sigma^x_{i+1})_{n+1}](   P_k\otimes P_{j_1} \otimes  P_{j_2}\,\otimes...\otimes P_{j_n})\right]          \]
\normalsize
$ \,\,= \frac{1}{2}\mu( j_1,j_2,...,j_n ) $
\[  \,\,\,\,- \frac{1}{2^{n+1}}\text{Tr}\left[\prod_{i=2}^{n} [\, I^{\otimes n+1} \,-\, (\sigma^x_{i}\otimes \sigma^x_{i+1})_{n+1}](   \sigma^x(P_k)\otimes \sigma^x(P_{j_1}) \otimes  P_{j_2}\,\otimes...\otimes P_{j_n})\right]          \]
$ \,\,= \frac{1}{2}\mu( j_1,j_2,...,j_n ) $
\footnotesize
$$  \,\,- \frac{1}{2^{n+1}}\text{Tr}\left[I^{\otimes n+1}\circ\prod_{i=3}^{n} [\, I^{\otimes n+1} \,-\, (\sigma^x_{i}\otimes \sigma^x_{i+1})_{n+1}](   \sigma^x(P_k)\otimes \sigma^x(P_{j_1}) \otimes  P_{j_2}\,\otimes...\otimes P_{j_n})\right]         $$
\[  \,\,\,\,+ \frac{1}{2^{n+1}}\text{Tr}\left[(\sigma^x_{2}\otimes \sigma^x_{3})_{n+1}\circ\prod_{i=3}^{n} [\, I^{\otimes n+1} \,-\, (\sigma^x_{i}\otimes \sigma^x_{i+1})_{n+1}](   \sigma^x(P_k)\otimes \sigma^x(P_{j_1}) \,\otimes...\otimes P_{j_n})\right]          \]
\normalsize
$ \,\,= \frac{1}{2}\mu( j_1,j_2,...,j_n ) - \frac{\beta_k \beta_{j_1}}{4}\mu(j_2,...,j_n)$
\[  \,\,\,\,+ \frac{1}{2^{n+1}}\text{Tr}\left[\prod_{i=3}^{n} [\, I^{\otimes n+1} \,-\, (\sigma^x_{i}\otimes \sigma^x_{i+1})_{n+1}](   \sigma^x(P_k)\otimes P_{j_1} \otimes  \sigma^x(P_{j_2})\,\otimes...\otimes P_{j_n})\right]         . \]
\normalsize
Using similar computations we get

\begin{align*}
\mu(k,&j_1,...,j_n)
 =\frac{1}{2}\mu( j_1,j_2,...,j_n ) - \frac{\beta_k\beta_{j_1}}{2^{2}}\mu(j_2,...,j_n) +\frac{\beta_k\beta_{j_2}}{2^{3}}\mu(j_3,...,j_n)\\
&+...+\frac{(-1)^{n-2}\beta_k\beta_{j_{n-2}}}{2^{n-1}}\mu(j_{n-1},j_n)+ \\
& +\,\frac{(-1)^{n-1}}{2^{n+1}}\,\text{Tr} \left\{ \begin{array}{l}
(\, + \, (\underbrace{I \otimes  ... \otimes  I}_{n+1})\,-\,
(I \otimes I  \otimes...\otimes I \otimes \sigma^x \otimes \sigma^x)\,)\,\,\circ
\\
( \sigma^x P_k \otimes  P_{j_1} \otimes ...\otimes \sigma^x P_{j_{n-1}} \otimes P_{j_n}))\end{array}\right\} \\
&= \frac{1}{2}\mu( j_1,j_2,...,j_n ) - \frac{\beta_k\beta_{j_1}}{2^{2}}\mu(j_2,...,j_n) +\frac{\beta_k\beta_{j_2}}{2^{3}}\mu(j_3,...,j_n)-...\\
&...+\frac{(-1)^{n-2}\beta_k\beta_{j_{n-2}}}{2^{n-1}}\mu(j_{n-1},j_n) +\frac{(-1)^{n-1} \beta_{k} \beta_{j_{n-1}}}{2^{n+1}}+ \frac{(-1)^{n} \beta_{k} \beta_{j_n}}{2^{n+1}}.
\end{align*}

\end{proof}

\medskip

\begin{proposition}  \label{nmix1} For any $n\geq 2$ we have that
$$\sum_{j_2,..,j_n}\mu (a ,b , j_2,...,j_n,c,d)=\mu(a,b)\mu(c,d )+\frac{(-1)^{n}( \beta_{b}-\beta_a) (\beta_{c}-\beta_{d})}{2^{4}}. $$
Particularly, $\mu$ is not mixing.
\end{proposition}

\begin{proof} From Theorem \ref{medidacilindro} we get
$$   \mu (a ,b , j_2,...,j_n,c,d) \,=\,\,\, \frac{\mu( b,j_2,...,j_n,c,d )}{2} - \frac{\beta_{a}\, \beta_{b}}{2^{2}} \, \mu (j_2,...,j_n,c,d) + \,$$
 $$
+\frac{\beta_{a}\beta_{j_2}}{2^{3}} \mu( j_3,j_4,...,j_n,c,d ) - \frac{\beta_{a} \beta_{j_3}}{2^{4}} \mu( j_4,j_5,...,j_n,c,d )+...$$
$$ +\frac{(-1)^{n} \beta_{a} \beta_{j_{n}}}{2^{n+1}} \mu(c,d)+ \frac{(-1)^{n+1} \beta_{a} (\beta_{c}-\beta_{d})}{2^{n+3}} . $$
Note that  $$\sum_{j_2=1}^2\frac{\beta_{a}\beta_{j_2}}{2^{3}} \mu( j_3,j_4,...,j_n,c,d )=0,$$
because $\beta_2=-\beta_1$ and ${j_2}$ only appears in $\beta_{j_2}$. In general, for $k=2,...,n$
$$\sum_{j_k=1}^2\frac{\beta_{a}\beta_{j_k}}{2^{k+1}}=0.$$
Therefore, we will separate with - the notation $\sum_{k=2}^n\beta_{j_k}(...)$ - all terms that contain $\frac{\beta_{a}\beta_{j_k}}{2^{k+1}}$.
In this way
$$   \mu (a ,b , j_2,...,j_n,c,d) \,=\frac{\mu( b,j_2,...,j_n,c,d )}{2} - \frac{\beta_{a}\, \beta_{b}}{2^{2}} \, \mu (j_2,...,j_n,c,d)+$$
$$ + \frac{(-1)^{n+1} \beta_{a} (\beta_{c}-\beta_{d})}{2^{n+3}}+\sum_{k=2}^n\beta_{j_k}(...)= $$

$$=\frac{1}{2} \bigg[\frac{\mu( j_2,...,j_n,c,d )}{2} - \frac{\beta_{b}\, \beta_{j_2}}{2^{2}} \, \mu (j_3,...,j_n,c,d)\bigg] + $$

$$+\frac{1}{2} \bigg[\frac{\beta_{b}\beta_{j_3}}{2^{3}} \mu( j_4,...,j_n,c,d ) - \frac{\beta_{b} \beta_{j_4}}{2^{4}} \mu(j_5,...,j_n,c,d ) \bigg]+...$$
$$+\frac{1}{2} \bigg[ \frac{(-1)^{n-1} \beta_{b} \beta_{j_{n}}}{2^{n}} \mu(c,d)+ \frac{(-1)^{n} \beta_{b} (\beta_{c}-\beta_{d})}{2^{n+2}} \bigg]-$$
$$ - \frac{\beta_{a}\, \beta_{b}}{2^{2}} \, \mu (j_2,...,j_n,c,d) +  \frac{(-1)^{n+1} \beta_{a} (\beta_{c}-\beta_{d})}{2^{n+3}}+  \sum_{k=2}^n\beta_{j_k}(...)= $$

 $$=\frac{\mu( j_2,...,j_n,c,d )}{2^2}- \frac{\beta_{a}\, \beta_{b}}{2^{2}} \, \mu (j_2,...,j_n,c,d) +$$

 $$  + \frac{(-1)^{n} \beta_{b} (\beta_{c}-\beta_{d})}{2^{n+3}} + \frac{(-1)^{n+1} \beta_{a} (\beta_{c}-\beta_{d})}{2^{n+3}}+  \sum_{k=2}^n\beta_{j_k}(...)=$$
 $$=\mu(a,b)\mu( j_2,...,j_n,c,d )+\frac{(-1)^{n}( \beta_{b}-\beta_a) (\beta_{c}-\beta_{d})}{2^{n+3}}+  \sum_{k=2}^n\beta_{j_k}(...). $$

 Then, we get
 $$
\mu (a ,b , j_2,...,j_n,c,d) \,=$$
$$
\mu(a,b)\mu( j_2,...,j_n,c,d )+\frac{(-1)^{n}( \beta_{b}-\beta_a) (\beta_{c}-\beta_{d})}{2^{n+3}}+  \sum_{k=2}^n\beta_{j_k}(...).
$$

We know that $\mu$ is stationary, so
\[\sum_{j_2,...,j_n}\mu( j_2,...,j_n,c,d ) = \mu(c,d).\]
It follows that
 \[\sum_{j_2,...,j_n}\mu (a ,b , j_2,...,j_n,c,d) \,=\mu(a,b)\mu(c,d )+\frac{(-1)^{n}( \beta_{b}-\beta_a) (\beta_{c}-\beta_{d})}{2^4}.\]

\bigskip

From the above   for $(a,b)=(1,2)$ and $(c,d)=(2,1)$ we finally get that
$$ \lim_{n \to \infty} \mu (\sigma^{-2n} (c,d) \cap (a,b)\,) \neq \mu(a,b)\, \mu(c,d) ,$$
and, this shows that $\mu$ is not mixing (see \cite{Man} or \cite{PoYu}). This also implies that  for some continuous functions there is no decay of correlation to $0$.
%%%%%%%%%%%%%%%%%%%%%%%%%%%%%%%

\end{proof}

\medskip

 \begin{proposition}\label{recursive1}
 	For $n\geq 1$,
 	$$\mu (k_0 , k_{1} , ...,k_n)=\frac{1}{2}\bigg(1-\frac{\beta_{k_0}}{\beta_{k_1}}\bigg)\mu (k_1 , ...,k_{n}) +\frac{\beta_{k_0}}{2} \bigg(\frac{1}{2\beta_{k_{1}}}-\frac{\beta_{k_{1}}}{2}\bigg)\mu (k_2,...,k_{n}).$$
  \end{proposition}

 \begin{proof}
 The cases $n=1,2$ correspond to the equations
 \[ \mu (k_0 , k_{1})=\frac{1}{2}\bigg(1-\frac{\beta_{k_0}}{\beta_{k_1}}\bigg)\mu (k_1) +\frac{\beta_{k_0}}{2} \bigg(\frac{1}{2\beta_{k_{1}}}-\frac{\beta_{k_{1}}}{2}\bigg), \]
 and,
 \[\mu (k_0 , k_{1} ,k_2)=\frac{1}{2}\bigg(1-\frac{\beta_{k_0}}{\beta_{k_1}}\bigg)\mu (k_1 , k_{2}) +\frac{\beta_{k_0}}{2} \bigg(\frac{1}{2\beta_{k_{1}}}-\frac{\beta_{k_{1}}}{2}\bigg)\mu (k_2) ,\]
 which can be checked directly from example \ref{example} and definition $\mu(k)=1/2, \,k=1,2$.
 	
 For $n\geq 2$, from (\ref{eq0}) we get the equations
 \begin{align*}
 \frac{2\mu(k_0,k_1,...,k_n)}{\beta_{k_0}} &= \frac{\mu(k_1,...,k_n)}{\beta_{k_0}} + \sum_{i=1}^{n-2}\frac{(-1)^i\beta_{k_i}\mu(k_{i+1},...,k_n)}{2^i}\\
 &\,\,\,+ (-1)^n\frac{\beta_{k_n}-\beta_{k_{n-1}}}{2^n} \\
 &=\frac{\mu(k_1,...,k_n)}{\beta_{k_0}} - \frac{\beta_{k_1}\mu(k_2,...,k_n)}{2}\\
 &\,\,\,+ \sum_{i=2}^{n-2}\frac{(-1)^i\beta_{k_i}\mu(k_{i+1},...,k_n)}{2^i}
+ (-1)^n\frac{\beta_{k_n}-\beta_{k_{n-1}}}{2^n},
 \end{align*}	
 and,
 \begin{align*} 	
 	\frac{\mu(k_1,...,k_n)}{\beta_{k_1}} &= \frac{\mu(k_2,...,k_n)}{2\beta_{k_1}}+\sum_{i=1}^{n-3}\frac{(-1)^i\beta_{k_{i+1}}\mu(k_{i+2},...,k_n)}{2^{i+1}} \\
 	&\,\,\,\,+(-1)^{n-1}\frac{\beta_{k_n}-\beta_{k_{n-1}}}{2^n} \\
 	&=\frac{\mu(k_2,...,k_n)}{2\beta_{k_1}}+\sum_{j=2}^{n-2}\frac{(-1)^{j-1}\beta_{k_{j}}\mu(k_{j+1},...,k_n)}{2^{j}} \\
 	&\,\,\,\,+(-1)^{n-1}\frac{\beta_{k_n}-\beta_{k_{n-1}}}{2^n} . 	
 \end{align*} 	
Then,
$$
\frac{2\mu(k_0,k_1,...,k_n)}{\beta_{k_0}} +  	\frac{\mu(k_1,...,k_n)}{\beta_{k_1}}$$ $$
= 	\left(\frac{\mu(k_1,...,k_n)}{\beta_{k_0}} - \frac{\beta_{k_1}\mu(k_2,...,k_n)}{2}\right) + \left(\frac{\mu(k_2,...,k_n)}{2\beta_{k_1}} \right).
$$ 	
Therefore,
$$\mu (k_0 , k_{1} , ...,k_n)=\frac{1}{2}\bigg(1-\frac{\beta_{k_0}}{\beta_{k_1}}\bigg)\mu (k_1 , ...,k_{n}) +\frac{\beta_{k_0}}{2} \bigg(\frac{1}{2\beta_{k_{1}}}-\frac{\beta_{k_{1}}}{2}\bigg)\mu (k_2,...,k_{n}).$$
 	
 \end{proof}

\end{document}